\documentclass[11pt,leqno,twoside]{amsart}
\usepackage{amssymb,amsmath,amsthm,color}
\usepackage[T1]{fontenc}
\usepackage{indentfirst}
\usepackage{cite,enumitem,graphicx}
\usepackage[colorlinks=true,urlcolor=blue,
citecolor=red,linkcolor=blue,linktocpage,pdfpagelabels,
bookmarksnumbered,bookmarksopen]{hyperref}
\usepackage[english]{babel}
\usepackage[left=2.61cm,right=2.61cm,top=2.72cm,bottom=2.72cm]{geometry}
\usepackage[hyperpageref]{backref}

\makeatletter
\providecommand\@dotsep{5}
\def\listtodoname{List of Todos}
\def\listoftodos{\@starttoc{tdo}\listtodoname}
\makeatother

\numberwithin{equation}{section}
\newtheorem{theorem}{Theorem}[section]
\newtheorem{proposition}[theorem]{Proposition}
\newtheorem{definition}[theorem]{Definition}
\newtheorem{lemma}[theorem]{Lemma}

\newtheorem{corollary}[theorem]{Corollary}
\newtheorem{remark}{Remark}

\def\R{\mathbb{R}}

\def\R {{\rm I}\hskip -0.85mm{\rm R}}

\title[The breakdown of the maximum principle]
{Characterizing the breakdown of the maximum principle in a class of nonlocal elliptic problems}

\author[J. R. Santos Jr.]{Jo\~ao R. Santos J\'unior}
\author[S. Sastre-G\'omez]{Silvia Sastre-G\'omez}
\author[D. A. Souza]{Diego A. Souza}

\address[J. R. Santos Jr.]
{\newline\indent Faculdade de Matem\'atica
\newline\indent
Instituto de Ci\^{e}ncias Exatas e Naturais
\newline\indent
Universidade Federal do Par\'a
\newline\indent
Avenida Augusto corr\^{e}a 01, 66075-110, Bel\'em, PA, Brazil}
\email{\href{mailto: joaojunior@ufpa.br }{joaojunior@ufpa.br}}

\address[S. Sastre-G\'omez]
{\newline\indent Departamento de Ecuaciones Diferenciales y Análisis Numérico 
\newline\indent 
and
\newline\indent Instituto de Matemáticas de la Universidad de Sevilla,
\newline\indent
Universidad de Sevilla
\newline\indent
Campus Reina Mercedes, 41012, Seville, Spain}
\email{\href{mailto: ssastre@us.es}{ssastre@us.es}}

\address[D. A. Souza]
{\newline\indent Departamento de Ecuaciones Diferenciales y Análisis Numérico 
\newline\indent 
and
\newline\indent Instituto de Matemáticas de la Universidad de Sevilla,
\newline\indent
Universidad de Sevilla
\newline\indent
Campus Reina Mercedes, 41012, Seville, Spain}
\email{\href{mailto: desouza@us.es}{desouza@us.es}}

\thanks{Jo\~ao R. Santos was partially
supported by CNPq/Brazil 304340/2025-1, Brazil. S. Sastre-G\'omez is supported by PID2023-149509NB-I00 funded by MICIU and by the IMUS–María de Maeztu grant CEX2024-001517-M (Apoyo a Unidades de Excelencia María de Maeztu) funded by MICIU. Diego A. Souza
is supported by the project PID2024-158206NB-I00, funded by MICIU/AEI/10.13039/501100011033
and FEDER, UE; by the IMUS–María de Maeztu grant CEX2024-001517-M (Apoyo a Unidades de
Excelencia María de Maeztu), funded by MICIU/AEI/10.13039/501100011033; and by the Consolidación
Investigadora project CNS2024-154725, also funded by MICIU/AEI/10.13039/501100011033. }
\thanks{{\it Data availability statement}: Data sharing not applicable to this article as no datasets were generated or analysed during the current study.
{\it Conflict of interest statement}: The authors declare that they have no conflict of interest.}


\begin{document}
\begin{abstract}

In this paper, we consider a nonlocal elliptic equation involving a nonnegative parameter $a$ that quantifies the strength of the nonlocal term, posed in a smooth bounded domain with homogeneous Dirichlet boundary conditions and a given nonnegative source term $f$.
This class of problems arises in several applications, including the modeling of thermistor behaviour and population dynamics with crowding effects. We focus on the breakdown of the maximum principle as the parameter $a$ increases.
We characterize the two types 
of degeneracy that may occur, possibly simultaneously, at the critical parameter \(a^*\): the solution
may vanish at an interior point of the domain, possibly forming a dead-core region, 
or its normal derivative may vanish on the boundary. 
In particular, we obtain sufficient conditions 
for the emergence of dead-core solutions, 
thereby answering an open question raised in the recent literature. Several one- and two-dimensional examples are constructed 
to illustrate each possible scenario. 
Additionally, we study an optimization problem 
motivated by biological considerations, 
where the goal is to maximize the critical threshold \(a^*\) 
over an admissible class of resource distributions. 
For sources with prescribed total mass, a uniform upper bound, and an interior-contact condition, we show that the critical threshold is unbounded and that the corresponding total population can be arbitrarily small. We also establish the existence of a maximizer in a class with a prescribed positive lower bound on the total population, a uniform bound on the sources, and the same interior-contact condition.

\end{abstract}

\maketitle


\bigskip

\begin{center}
\begin{minipage}{12cm}
\tableofcontents
\end{minipage}
\end{center}

\bigskip

\section{Introduction}

Elliptic equations involving nonlocal terms of integral type 
have attracted considerable attention in recent decades, 
due both to their intrinsic mathematical interest 
and to their relevance in a variety of applications. 

Let \(\Omega \subset \mathbb{R}^N\) be a bounded smooth domain, denote by $\eta(x)$ the outward unit normal vector at a point $x\in \partial\Omega$ and by $\partial_\eta$ the outward normal derivative. Problems of the form
\begin{equation}
\left\{
\begin{array}{lll}
-\Delta u + a \displaystyle\int_\Omega u \,dy = f &\text{in}& \Omega,\\[1.2em]
u = 0 &\text{on} &\partial\Omega,
\end{array}
\right.
\end{equation}
where \(a\in \mathbb{R}\) is a parameter and \(f\) is a given function, 
 arise in contexts ranging from microsensor thermistor modelling 
to population dynamics. 
In the thermistor setting, as introduced by Allegretto et al. \cite{Allegretto1994} 
and further developed in \cite{AllegrettoBarabanova}, 
the function \(u\) represents the temperature distribution, 
the integral term corresponds to heat loss from the device surface 
to the surrounding gas, 
and the parameter \(a\) is related to gas pressure and other physical factors. 
The same equation has been employed to model other phenomena, 
including the steady-state behaviour of electric wires 
under Neumann boundary conditions \cite{Chafee1981} 
and various nonlocal diffusion processes.

In the biological interpretation, 
which is the main motivation for the present work, 
\(u(x)\) denotes the population density of a species at a point \(x\) 
in a habitat \(\Omega\). 
The Dirichlet boundary condition \(u=0\) on \(\partial\Omega\) 
means that the region outside \(\Omega\) is hostile, 
so that individuals cannot survive at the boundary. 
The nonlocal term  
represents a crowding effect: 
the total population influences the local dynamics 
through the parameter \(a\), 
which measures the intensity of this global interaction. 
This type of modelling, 
where the spatial distribution is affected by the total population, 
is natural in ecological systems 
and has been discussed in the works of Chipot and Rodrigues \cite{ChipotRodrigues1992} 
and Furter and Grinfeld \cite{FurterGrinfeld1989}. Further developments concerning positivity, maximum principles, spectral properties, and the sign of solutions for this class of problems and related nonlocal elliptic operators can be found in \cite{FreitasSweers,AlvesCorreaSantosJunior,BahiaPimentaSantos,CintraAnjos}. Existence and multiplicity results for semilinear and quasilinear problems, systems, singular equations, and equations involving subcritical, critical, and supercritical nonlinearities have been obtained in \cite{CabadaCorrea,CorreaLimaLima,dosSantosLimaLima,dosSantosLimaLimaStrongSingularity,deLimaNobregaTavares,GuefaifiaBoulaaras,dosSantosCorreaSilva}; see also the recent overview \cite{CorreaLimaLimaOverview}. Sharp estimates and saturation phenomena for a related nonlocal eigenvalue problem were investigated in \cite{BrandoliniFreitasNitschTrombetti}, while an earlier integro-differential formulation in the ordinary differential setting can be found in \cite{Catchpole}.

A fundamental question for this class of problems 
is whether the solution \(u_a\) remains nonnegative 
for nonnegative source terms \(f\), 
i.e., whether a maximum principle holds. 
The first systematic study in this direction 
was carried out by Allegretto and Barabanova \cite{AllegrettoBarabanova}. 
In that work, the authors proved the existence 
of a critical parameter \(a_0 > 0\) 
such that for \(0\le a < a_0\), 
the solution is positive in \(\Omega\) 
and has negative outward normal derivative on \(\partial\Omega\). 
They also provided an explicit one-dimensional example 
with \(f(x)=\sin(\pi x)\) on \(\Omega=(0,1)\), 
showing that for sufficiently large \(a\), 
the solution changes sign near the endpoints. 
Based on this observation, 
they conjectured that in higher dimensions, 
for large values of the parameter, 
the solution should also change sign. 
However, the nonconstructive nature of the methods employed in \cite{AllegrettoBarabanova} 
prevented a complete answer to this question.

The conjecture of Allegretto and Barabanova 
was recently resolved by Bahia, Pimenta and Santos Júnior \cite{BahiaPimentaSantos}. 
In that work, the authors considered the problem 
with \(a\ge 0\) and \(f\) bounded, nonnegative, and nonconstant, 
and established the existence of a critical parameter 
\(a^*(f)>0\), defined by
\[
a^*(f):=\sup\{a\ge 0 : u_a>0 \text{ in } \Omega 
\text{ and } \partial_\eta u_a<0 \text{ on } \partial\Omega\},
\]
where $\eta$ denotes the outward unit normal vector.
They proved that:
\begin{itemize}
\item for \(0\le a < a^*\), the maximum principle holds, 
with \(u_a>0\) in \(\Omega\) and \(\partial_\eta u_a<0\) on \(\partial\Omega\);
\item at \(a=a^*\), the solution remains nonnegative 
but degenerates in one of two ways: 
either there exists an interior point \(x_0\in\Omega\) 
such that \(u_{a^*}(x_0)=0\), 
or there exists a boundary point \(x_1\in\partial\Omega\) 
such that \(\partial_\eta u_{a^*}(x_1)=0\);
\item for \(a>a^*\), the solution changes sign in \(\Omega\).
\end{itemize}
This result provided a complete positive answer 
to the Allegretto--Barabanova conjecture in arbitrary dimension. 
In particular, it revealed the existence of a critical threshold \(a^*\), 
below which the model produces a strictly positive population profile. Above this threshold, the solution changes sign and therefore no longer represents a physically admissible population density. 
From a biological standpoint, 
the parameter \(a^*\) represents the threshold intensity of the crowding effect at which hostile zones first emerge. Thus, for \(a<a^*\), the population can withstand the effects of crowding without the formation of hostile zones, while \(a^*\) marks the onset of this phenomenon. Understanding the precise nature of the degeneracy at \(a=a^*\),
whether it occurs in the interior or on the boundary,
is therefore of both mathematical and ecological significance.

However, the work of Bahia, Pimenta and Santos Júnior \cite{BahiaPimentaSantos} 
left an important question open. 
While the authors characterized the critical parameter 
and the two possible types of maximum principle breakdown, 
they did not provide sufficient conditions on \(f\) 
that would guarantee the occurrence of interior degeneracy, 
i.e., the formation of a \emph{dead core} or \emph{dead zone}. 
A dead core is an open set \(\Omega_*\subset\subset \Omega\) 
where the solution vanishes identically, 
representing a hostile zone within the habitat 
where the population is extinct. 
This phenomenon is of particular biological interest, 
as it corresponds to the emergence of uninhabitable regions 
inside the domain, 
surrounded by areas where the population may still survive. 
The existence of dead-core solutions 
was explicitly posed as an open problem 
in Section~5.1 of \cite{BahiaPimentaSantos}.

The present paper addresses this open question. 
Our main result, stated in Theorem \ref{teo:interior}, 
gives a sharp characterization of the two types of degeneracy 
at the critical level in terms of the ratio \(u_f/e\), 
where \(e\) is the unique solution of 
\(-\Delta e=1\) in \(\Omega\), \(e=0\) on \(\partial\Omega\), 
and \(u_f\) is the solution of the Dirichlet problem 
\(-\Delta u_f=f\). 
Specifically, we prove that interior degeneracy 
(i.e., \(u_{a^*}(x_0)=0\) for some \(x_0\in\Omega\)) 
occurs if and only if the infimum of \(u_f/e\) over \(\Omega\) 
is attained at an interior point. 
Moreover, the critical parameter admits the explicit formula
\[
a^*(f)=\frac{\displaystyle\lambda^*(f)}{\displaystyle\int_\Omega f e\,dy 
- \lambda^*(f)\int_\Omega e\,dy},
\]
where 
\(\lambda^*(f):=\inf_{\Omega} u_f/e\). 
This characterization provides a direct formula for the critical parameter in terms of the source and the associated local solution.

Moreover, in Corollary \ref{suficiente} we establish a sufficient condition 
for the existence of dead-core solutions: 
if there exists an open set \(\Omega_0\subset\Omega\) 
where the ratio \(u_f/e\) is constant and attains its minimum, 
then \(\Omega_0\) is a dead zone of \(u_{a^*}\). 
This provides a concrete criterion for the emergence of dead cores in this class of problems.

The power of our characterization is demonstrated 
through a series of explicit examples in both one and two dimensions. 
In Section \ref{exam}, we construct four one-dimensional examples 
illustrating each possible scenario: 
(i) interior and boundary degeneracy simultaneously; 
(ii) boundary degeneracy only; 
(iii) interior degeneracy only; 
and (iv) the emergence of a dead zone in the interior. 
We also provide a two-dimensional example on the unit ball, 
showing that dead-core solutions occur in higher dimensions as well. 
These examples resolve the open issue raised in \cite{BahiaPimentaSantos} 
regarding the existence of functions \(f\) 
for which the maximum principle breaks down 
through the formation of dead zones.

A key technical innovation of our work 
is the representation formula for the solution \(u_a\) 
given in Lemma~\ref{ok}:
\[
u_a = u_f - \lambda(a,f,\Omega)e,
\]
where
\[
\lambda(a,f,\Omega):=\frac{\displaystyle a\int_\Omega f e\,dy}{\displaystyle 1+a\int_\Omega e\,dy}.
\]
This representation, which differs from the one used in \cite{BahiaPimentaSantos}, 
allows us to compare \(u_a\) with the difference \(u_f-\lambda e\) 
and to relate the degeneracy of \(u_a\) at the critical level 
to the attainment of the infimum of the ratio \(u_f/e\). 
The technical difficulties involved in this analysis 
include the delicate interplay between the nonlocal term 
and the elliptic operator, 
as well as the need to handle functions \(f\) 
that are only in \(L^\infty(\Omega)\) 
and not necessarily smooth.

In addition to the characterization 
of the maximum principle breakdown, 
we study an optimization problem 
motivated by the biological interpretation of the model. 
In a conservation context, 
it is natural to ask whether, 
within an admissible class of resource distributions \(f\), 
one can maximize the critical threshold \(a^*(f)\). 
Biologically, a larger value of \(a^*\) 
means that the population can tolerate a higher crowding effect 
before a hostile region appears, 
which would favour the persistence of the species. 
More precisely, in Section~5 we consider 
the variational problem
\[
\sup_{f\in \mathcal{A}_{\delta,C,M}} a^*(f),
\]
where \(\mathcal{A}_{\delta,C,M}\) is the class of functions \(f\) 
satisfying \(0\le f\le M\), \(\int_\Omega f\,dy=C\), 
and the interior-contact condition 
(ensuring that the dead-core, if it occurs, 
does so away from the boundary). 
We investigate, in Theorem \ref{ttteo}, 
the case where we have a prescribed total mass condition 
\(\int_\Omega f\,dy=C\). 
We show that, even under the mass constraint, the uniform bound, and the interior-contact condition, 
the supremum of \(a^*(f)\) over the admissible class is infinite. 
This indicates that the total mass constraint 
is not sufficient to prevent degeneracy phenomena, 
and that the spatial distribution of resources 
plays a crucial role in determining the critical threshold. {\color{black}That is why in Section \ref{Prtotpop} we replace the total mass condition by a total population constraint and we prove the existence of at least one maximizer of the supremum of $a^*(f)$.}

We also prove in Theorem \ref{tteo} 
that this supremum is attained, 
i.e., there exists an optimal resource distribution 
maximizing the critical threshold. 
This result has a clear biological interpretation: 
under a prescribed minimum viable population 
and a bound on available resources, 
there is an optimal resource profile 
that maximizes the resilience of the species 
to crowding effects.

The paper is organized as follows. 
In Section \ref{prel}, we recall the preliminary results 
needed for our analysis, 
including the definition of the critical parameter 
and the basic representation formula for the solution. 
In Section \ref{charact}, we establish our main characterization theorem 
for the breakdown of the maximum principle 
and provide an explicit formula for the critical parameter. 
Section \ref{exam} is devoted to explicit examples 
in one and two dimensions, 
illustrating each possible scenario. 
Finally, in Section \ref{otimiz}, 
we study the optimization problem involving \(a^*\), 
first under a prescribed total mass constraint 
and then under a prescribed total population constraint, 
and we discuss the biological implications of our results.


\section{Preliminary Results}\label{prel}

Let \(\Omega \subset \R^N\) be a smooth bounded domain. We consider the nonlocal problem
\begin{equation}\label{eq:P01}
\left\{
\begin{array}{lll}
-\Delta u + a \displaystyle\int_\Omega u \,dy = f &\text{in}& \Omega,\\[1.2em]
u = 0 &\text{on} &\partial\Omega,
\end{array}
\right.
\end{equation}
with \(a \in \R\) and \(f \in L^\infty(\Omega)\), \(f \geq 0\) and non-constant.

\medskip

Define the \textit{torsion function} \(e\) as the unique solution of
\begin{equation}\label{solut}
-\Delta e = 1 \quad \text{in } \Omega,\qquad e = 0 \quad \text{on } \partial\Omega.
\end{equation}
It is well known that \(e > 0\) in \(\Omega\) and \(\partial_\eta e < 0\) on \(\partial\Omega\).

\medskip

For a given \(f\in L^\infty(\Omega)\), let $u_f\in C^{1, \alpha}(\overline{\Omega})$, for some $\alpha\in (0, 1)$, denote the unique solution of the Dirichlet problem
\begin{equation}\label{solu}
-\Delta u_f = f \quad \text{in } \Omega,\qquad u_f = 0 \quad \text{on } \partial\Omega.
\end{equation}
The solution operator \(f \mapsto u_f\) is linear, continuous and compact from \(L^\infty(\Omega)\) into \(C(\overline{\Omega})\).

\begin{proposition}
\label{prop:sol}
For every $$\displaystyle a \neq -\frac{1}{\displaystyle\int_\Omega e(y)\,dy},$$ the unique solution \(u_a\) of \eqref{eq:P01} is given by
\begin{equation}\label{for1}
u_a = \frac{\displaystyle\int_{\Omega}fe dy}{\displaystyle\int_{\Omega}e dy\left(1+a\int_{\Omega}e dy\right)}e+P_{X}u_f, \qquad x \in \Omega,
\end{equation}
where $P_Xu_f$ is the orthogonal projection of $u_f$ onto
$$X=\left\{u\in H_{0}^1(\Omega): \displaystyle\int_{\Omega}u dy=0\right\}$$
with respect to the inner product $(u,v)_{H_0^1}:=\int_\Omega\nabla u\cdot\nabla v\,dy$.
\end{proposition}

\begin{proof}
Set $I_a:=\int_\Omega u_a\,dy$. Since $-\Delta e=1$, the equation in \eqref{eq:P01} gives
\[
u_a=u_f-aI_a e.
\]
Integrating this identity and using $\int_\Omega u_f\,dy=\int_\Omega fe\,dy$, we obtain
\[
I_a=\frac{\int_\Omega fe\,dy}{1+a\int_\Omega e\,dy}.
\]
Moreover, $X$ is the orthogonal complement of $\operatorname{span}\{e\}$ in $H_0^1(\Omega)$, because
\[
(u,e)_{H_0^1}=\int_\Omega u\,dy.
\]
Thus
\[
u_f=P_Xu_f+\frac{\int_\Omega fe\,dy}{\int_\Omega e\,dy}e,
\]
and substitution in $u_a=u_f-aI_ae$ yields \eqref{for1}.
\end{proof}
\begin{remark}
    From identity \eqref{for1}, we directly deduce the following implicit representation of the parameter $a$
\[
a= \frac{\displaystyle1}{\displaystyle\int_\Omega e(x)dy}\left(\frac{\displaystyle\int_{\Omega}f(x)e(x) dy}{\displaystyle\int_\Omega u_a(x) dy}-1\right).
\]
\end{remark}

From now on we restrict ourselves to the case \(a \geq 0\) and \(f\) nonnegative.

\begin{definition}[Critical parameter]\label{Defi1}
For \(f \in L^\infty(\Omega)\), \(f \geq 0\), we define
\[
a^*(f) = \sup\bigl\{ a \geq 0 : u_a > 0 \text{ in } \Omega \text{ and } \partial_\eta u_a < 0 \text{ on } \partial\Omega \bigr\}.
\]
If the supremum is unbounded we write \(a^\ast(f)=+\infty\).
\end{definition}

\begin{theorem}[Allegretto–Barabanova \cite{AllegrettoBarabanova}]
\label{thm:eta1}
There exists a constant \(\eta_1 > 0\), depending only on \(\Omega\) and on the ellipticity constants, such that for every \(0 \leq a < \eta_1\) and every nontrivial \(f \geq 0\), the solution \(u_a\) of \eqref{eq:P} is positive in \(\Omega\) and satisfies \(\partial_\eta u_a < 0\) on \(\partial\Omega\).
\end{theorem}

\begin{remark}\label{obs}
It is important to point out some consequences of previous results:

\begin{itemize}
\item [(i)] The authors of \cite{BahiaPimentaSantos} proved that $a^\ast(f)$ is a positive real number when $f$ satisfies the conditions of Definition \ref{Defi1} and is nonconstant. Moreover, as a consequence of Theorem \ref{thm:eta1}, \(a^*(f)\) is uniformly bounded from below for all nonnegative, nonconstant \(f\in L^\infty(\Omega)\); that is, there exists a positive constant $\eta_1$ such that
\[
a^*(f) \geq \eta_1 \quad \text{for all } f \geq 0,\ \mbox{$f$ non-constant}.
\]

\item [(ii)] When \(f\) is constant, the solution (and its normal derivative) never touches zero in $\Omega$ (on $\partial\Omega$, respectively). In fact, if \(f \equiv c > 0\) is constant, then \(u_f = c\,e\) and, as demonstrated in \cite{BahiaPimentaSantos}, from \eqref{for1}, we obtain 
$$
\displaystyle u_a = \frac{c}{\displaystyle1 + a\int_\Omega e(x) dy}\,e.
$$
Since \(e > 0\) in \(\Omega\), it follows that \(u_a > 0\) for every \(a \geq 0\). Hence \(a^*(c) = +\infty\).

\item [(iii)] 
    Notice that the critical parameter $a^*(f)$ is homogeneous of degree zero, i.e. $a^*(tf) = a^*(f)$ for any $t > 0$. 
    Indeed, let $u_a(f)$ denote the unique solution to \eqref{eq:P01} for a given parameter $a$ and source $f$. By linearity, $u_a(tf) = t u_a(f)$. Since $t > 0$, the sign and boundary behavior of the solution are strictly preserved. Therefore:
    \begin{itemize}
        \item If $a < a^*(f)$, then $u_a(tf) > 0$ in $\Omega$ with $\partial_\eta u_a(tf) < 0$ on $\partial\Omega$.
        \item If $a > a^*(f)$, then $u_a(tf)$ changes sign in $\Omega$.
        \item If $a = a^*(f)$, then $u_{a^*}(tf) \ge 0$ in $\Omega$ and either $u_{a^\ast}(tf) $ vanishes at some interior point of $\Omega$ or its normal derivative $\partial_\eta u_{a^\ast}(tf) $ vanishes at some boundary point of $\partial\Omega$.
    \end{itemize}
    Since $a^*(f)$ is the exact transition threshold for the solutions associated with the scaled source $tf$, we conclude that $a^*(tf) = a^*(f)$.
\end{itemize}
\end{remark}

\section{Characterization of Maximum Principle Breakdown}\label{charact}

In what follows, we denote

$$
\mathcal{B}=\{f \in L^\infty(\Omega) : f \geq 0 \}
$$
and 
$$
\mathcal{A} = \{ f \in \mathcal{B}: f \text{ non-constant} \}.
$$

\begin{lemma}\label{ok}
 Let $f\in\mathcal{A}$, $e$ solution of \eqref{solut} and $u_f$ solution of \eqref{solu}. Then,
\begin{equation}\label{chara}
    u_a(x) = u_f(x) - \lambda(a, f, \Omega)\, e(x), \quad \forall x\in \overline{\Omega},
\end{equation}
where $u_a$ is the solution of \eqref{eq:P01} and
\begin{equation}\label{pulo}
\lambda(a, f, \Omega):=\frac{\displaystyle a\int_\Omega u_f dy}{\displaystyle1 + a\int_\Omega e dy}=\frac{\displaystyle a\int_\Omega fe dy}{\displaystyle1 + a\int_\Omega e dy}.
\end{equation}
\end{lemma}
\begin{proof}
Set $I_a:=\int_\Omega u_a\,dy$. By linearity and the definition of $e$,
\[
u_a=u_f-aI_ae.
\]
Integrating and using $\int_\Omega u_f\,dy=\int_\Omega fe\,dy$, we obtain
\[
I_a=\frac{\int_\Omega fe\,dy}{1+a\int_\Omega e\,dy},
\]
which gives \eqref{chara} and \eqref{pulo}.
This is also consistent with the decomposition in Proposition \ref{prop:sol}.
\end{proof}

\begin{lemma}\label{smart}
Let $f\in\mathcal{A}$ and consider $\lambda(a, f, \Omega)$  as in \eqref{pulo}. Then the following properties about the continuous map $ a\in [0, +\infty)\mapsto \lambda(a,f,\Omega)\in \mathbb{R}$ hold:

\begin{enumerate}
    \item[$(i)$] \(\lambda(0,f,\Omega)=0\);

    \item[$(ii)$] the map
    \[
    a\mapsto \lambda(a,f,\Omega)
    \]
    is strictly increasing on \([0,+\infty)\);

    \item[$(iii)$] the map
    \[
    a\mapsto \lambda(a,f,\Omega)
    \]
    is strictly concave on \([0,+\infty)\);

    \item[$(iv)$]
    \[
    \lim_{a\to+\infty}\lambda(a,f,\Omega)
    =
    \frac{\displaystyle\int_\Omega f e\,dy}
    {\displaystyle\int_\Omega e\,dy}.
    \]
\end{enumerate}
\end{lemma}

\begin{proof}
Let
\[
C:=\int_\Omega f e\,dy
\quad\text{and}\quad
D:=\int_\Omega e\,dy.
\]
Since \(f\in\mathcal A\) and \(e>0\) in \(\Omega\), we have
\[
C>0
\quad\text{and}\quad
D>0.
\]
Hence,
\begin{equation}\label{lambdaCD}
\lambda(a,f,\Omega)=\frac{aC}{1+aD},
\qquad a\ge0.
\end{equation}

Property (i) is immediate, since
\[
\lambda(0,f,\Omega)=0.
\]

Now,
\[
\frac{d}{da}\lambda(a,f,\Omega)
=
\frac{C(1+aD)-aCD}{(1+aD)^2}
=
\frac{C}{(1+aD)^2}.
\]
Since \(C>0\), it follows that
\[
\frac{d}{da}\lambda(a,f,\Omega)>0
\qquad \forall a\ge0,
\]
which proves (ii).

Moreover,
\[
\frac{d^2}{da^2}\lambda(a,f,\Omega)
=
-\frac{2CD}{(1+aD)^3}.
\]
Since \(C,D>0\), we obtain
\[
\frac{d^2}{da^2}\lambda(a,f,\Omega)<0
\qquad \forall a\ge0.
\]
Thus, we have proved (iii).

Finally, dividing numerator and denominator in \eqref{lambdaCD} by \(a\), we get
\[
\lambda(a,f,\Omega)
=
\frac{C}{\frac1a+D}.
\]
Therefore,
\[
\lim_{a\to+\infty}\lambda(a,f,\Omega)
=
\frac{C}{D}
=
\frac{\displaystyle\int_\Omega f e\,dy}
{\displaystyle\int_\Omega e\,dy},
\]
which concludes the proof.
\end{proof}

By exploring the characterization \eqref{chara} of the solution to the problem \eqref{eq:P01} (different from that given in \eqref{for1}), we obtain deeper insight into the type of degeneracy exhibited by $u_a$ at the critical level $a=a^\ast$. The next theorem addresses the issue highlighted in \cite{BahiaPimentaSantos} of finding sufficient conditions on $f$ under which the maximum principle breaks down through an interior contact point. To this end, we first introduce the following positive real constants:
\begin{equation}\label{star}
    \lambda^{\ast}(f) := \inf_{\Omega} \frac{u_f}{e}
\end{equation}
and
\begin{equation}\label{twostars}
    \lambda^{\ast\ast}(f) := \inf_{ \partial\Omega} \frac{\partial_{\eta} u_f}{\partial_\eta e}.
\end{equation}

The quotient $u_f/e$ admits a positive continuous extension to $\overline\Omega$, whose boundary values are
\[
\frac{\partial_\eta u_f}{\partial_\eta e}.
\]
Indeed, $u_f,e\in C^{1,\alpha}(\overline\Omega)$, both functions vanish on $\partial\Omega$, and the Hopf lemma gives $\partial_\eta u_f<0$ and $\partial_\eta e<0$ on $\partial\Omega$. Consequently, compactness of $\overline\Omega$ gives $\lambda^\ast(f)>0$, and compactness of $\partial\Omega$ ensures that the infimum in \eqref{twostars} is attained.

Before presenting the main theorem, we state several important properties of $\lambda^\ast$ and $\lambda^{\ast\ast}$ in the following lemma.
\begin{lemma}\label{lem:lambda-equality}
Suppose that $f\in\mathcal{A}$. Then,
\begin{enumerate}
\item[$(i)$] $\lambda^{\ast}(f)\leq \lambda^{\ast\ast}(f)$;

\item[$(ii)$] $\displaystyle\lambda^{\ast}(f)< \frac{\displaystyle\int_{\Omega}fe dy}{\displaystyle\int_{\Omega}e dy}$;

\item[$(iii)$] $\displaystyle\lambda^{\ast\ast}(f)\leq \frac{1}{|\Omega|}\int_{\Omega}f dy$.

\end{enumerate}
\end{lemma}
\begin{proof}
(\textit{i}) Define $w := u_f - \lambda^\ast(f) e$. By the definitions of $\lambda^\ast(f)$ and $w$, we have
\[
w \geq 0 \quad \text{in } \Omega, \qquad w = 0 \quad \text{on } \partial\Omega.
\]
We claim that $w \not\equiv 0$. Indeed, if $w \equiv 0$, then $u_f = \lambda^\ast(f) e$ in $\Omega$. Applying the Laplacian operator to both sides yields $f = \lambda^\ast(f)$ in $\Omega$, which contradicts $f \in \mathcal{A}$.

Since $w \geq 0$ in $\Omega$ and $w = 0$ on $\partial\Omega$, the outward normal derivative satisfies
\[
\partial_\eta w \leq 0 \quad \text{on } \partial\Omega,
\]
which implies
\[
\partial_\eta u_f - \lambda^\ast(f) \partial_\eta e \leq 0 \quad \text{on } \partial\Omega.
\]
Since $\partial_\eta e < 0$ on $\partial\Omega$, dividing by $\partial_\eta e$, we obtain the inequality
\[
\lambda^\ast(f) \leq \frac{\partial_\eta u_f}{\partial_\eta e} \quad \text{on } \partial\Omega.
\]
Taking the infimum over $\partial\Omega$, we conclude that $\lambda^\ast(f) \leq \lambda^{\ast\ast}(f)$.

\medskip
(\textit{ii}) From the definition of $\lambda^\ast(f)$, we have
\begin{equation}\label{santacruz}
u_f - \lambda^\ast(f) e \geq 0 \quad \text{in } \Omega.
\end{equation}
\textbf{Claim:} There exists a measurable set $\Omega_0 \subset \Omega$ with positive measure $|\Omega_0| > 0$ such that $u_f - \lambda^\ast(f)e > 0$ in $\Omega_0$.

Indeed, if the claim were false, then $u_f = \lambda^\ast(f) e$ almost everywhere in $\Omega$. Applying the Laplacian operator again yields $f = \lambda^\ast(f)$ in $\Omega$, contradicting $f \in \mathcal{A}$.

Integrating \eqref{santacruz} over $\Omega$ and applying the claim, we obtain
\[
\int_{\Omega} u_f \, dy - \lambda^\ast(f) \int_{\Omega} e \, dy > 0.
\]
Multiplying \eqref{solut} by  $u_f$, integrating by parts, and thanks to Green's second identity and \eqref{solu}, we have $$\int_\Omega fe \, dy = \int_{\Omega} u_f \, dy$$ 
Then, replacing this relation yields the desired inequality.

\medskip
(\textit{iii}) By definition of $\lambda^{\ast\ast}(f)$ and since $\partial_\eta e < 0$ on $\partial\Omega$, we have
\[
\partial_\eta u_f \leq \lambda^{\ast\ast}(f) \partial_\eta e \quad \text{on } \partial\Omega.
\]
Integrating over $\partial\Omega$ and applying the Divergence Theorem yields
\[
-\int_{\Omega} f \, dy = \int_{\partial\Omega} \partial_\eta u_f \, dS \leq \lambda^{\ast\ast}(f) \int_{\partial\Omega} \partial_\eta e \, dS =  -\lambda^{\ast\ast}(f)\int_\Omega 1\,dy=-\lambda^{\ast\ast}(f) |\Omega|.
\]
Rearranging the inequality gives $$\lambda^{\ast\ast}(f) \leq \frac{1}{|\Omega|} \int_{\Omega} f \, dy.$$
\end{proof}
\begin{remark}\label{rmq:lambda}
As an immediate consequence of item (ii) of Lemma~\ref{lem:lambda-equality}, there exists $a>0$ such that $\lambda(a,f,\Omega)=\lambda^*(f)$. Indeed, setting
\[
a = \frac{\lambda^\ast(f)}{\displaystyle\int_\Omega f e \, dy - \lambda^\ast(f) \int_\Omega e \, dy},
\]
the strict inequality in item $(ii)$ of Lemma~\ref{lem:lambda-equality} ensures that $a > 0$. Hence, by \eqref{pulo}, we obtain $\lambda(a,f,\Omega)=\lambda^*(f)$.
\end{remark}
We are now in a position to state the main result.

\begin{theorem}\label{teo:interior}
Let $f \in \mathcal{A}$. Then the following statements hold:
\begin{itemize}
    \item [(i)] There exists $x_1 \in \partial\Omega$ such that $\partial_\eta u_{a^*}(x_1)=0$ if and only if $\lambda^{\ast\ast}(f) = \lambda^\ast(f)$.
    \item [(ii)] There exists $x_0 \in \Omega$ such that $u_{a^*}(x_0) = 0$ if and only if $\lambda^\ast(f)$ is attained at $x_0$.
\end{itemize}
\end{theorem}
\begin{proof}
For simplicity, we write $a^\ast = a^\ast(f)$ throughout this proof.

\medskip
(\textit{i}) Assume first that there exists $x_1 \in \partial\Omega$ such that $\partial_\eta u_{a^\ast}(x_1) = 0$. By \eqref{chara}, we have
\begin{equation}\label{deri}
u_{a^\ast} = u_f - \lambda(a^\ast, f, \Omega) e \quad \text{in } \overline{\Omega}.
\end{equation}
By definition of $a^\ast$, we have $\partial_\eta u_{a^\ast} \leq 0$ on $\partial\Omega$. Consequently,
\[
\frac{\partial_\eta u_f}{\partial_\eta e} \geq \lambda(a^\ast, f, \Omega) \quad \text{on } \partial\Omega.
\]
Taking the infimum over $\partial\Omega$ and using \eqref{twostars}, we obtain
\begin{equation}\label{ok5}
\lambda^{\ast\ast}(f) \geq \lambda(a^\ast, f, \Omega).
\end{equation}
On the other hand, since $\partial_\eta u_{a^\ast}(x_1) = 0$, it follows from \eqref{deri} and definition \eqref{twostars} of $\lambda^{\ast\ast}(f)$ that
\begin{equation}\label{ok6}
\lambda(a^\ast, f, \Omega) = \frac{\partial_\eta u_f(x_1)}{\partial_\eta e(x_1)} \geq \lambda^{\ast\ast}(f).
\end{equation}
Combining \eqref{ok5} and \eqref{ok6} yields
\begin{equation}\label{iphone}
\lambda^{\ast\ast}(f) = \lambda(a^\ast, f, \Omega),
\end{equation}
and the infimum in \eqref{twostars} is attained at $x_1$. Moreover, since $u_{a^\ast} \geq 0$ in $\Omega$, identity \eqref{deri} implies
\begin{equation}\label{Maio}
\lambda(a^\ast, f, \Omega) \leq \lambda^\ast(f).
\end{equation}
The first implication then follows from \eqref{iphone}, \eqref{Maio}, and item (\textit{i}) of Lemma~\ref{lem:lambda-equality}.

\medskip
Conversely, assume that $\lambda^{\ast\ast}(f) = \lambda^\ast(f)$. Let $a > 0$ be the parameter given by Remark~\ref{rmq:lambda}. By \eqref{chara} and the definition of $\lambda^{\ast\ast}(f)$, we have
\begin{equation}\label{ok9}
\partial_\eta u_a = \partial_\eta u_f - \lambda(a, f, \Omega) \partial_\eta e = \partial_\eta u_f - \lambda^{\ast\ast}(f) \partial_\eta e \leq 0 \quad \text{on } \partial\Omega.
\end{equation}
Since $\partial\Omega$ is compact, the infimum $\lambda^{\ast\ast}(f)$ is attained at some $x_1 \in \partial\Omega$, so
\begin{equation}\label{insta}
\partial_\eta u_a(x_1) = \partial_\eta u_f(x_1) - \lambda^{\ast\ast}(f) \partial_\eta e(x_1) = 0.
\end{equation}
Finally, the equality $\lambda(a, f, \Omega) = \lambda^\ast(f)$ implies
\begin{equation}\label{things}
u_a \geq 0 \quad \text{in } \Omega.
\end{equation}
Together, \eqref{ok9}, \eqref{insta}, and \eqref{things} describe the loss of the strict boundary inequality at $a$.
It remains to identify $a$ with the critical parameter.
Indeed, if $0\leq b<a$, then Lemma~\ref{smart} gives $\lambda(b,f,\Omega)<\lambda^\ast(f)$, and hence
\[
u_b=u_a+\bigl(\lambda^\ast(f)-\lambda(b,f,\Omega)\bigr)e>0
\quad\text{in }\Omega.
\]
Moreover,
\[
\partial_\eta u_b
=
\partial_\eta u_a
+\bigl(\lambda^\ast(f)-\lambda(b,f,\Omega)\bigr)\partial_\eta e<0
\quad\text{on }\partial\Omega.
\]
If $b>a$, then at the point $x_1$ in \eqref{insta},
\[
\partial_\eta u_b(x_1)
=
\partial_\eta e(x_1)
\bigl(\lambda^\ast(f)-\lambda(b,f,\Omega)\bigr)>0.
\]
Thus every $b<a$ belongs to the set in Definition~\ref{Defi1}, whereas no $b>a$ does, proving $a=a^\ast(f)$.

\medskip
(\textit{ii}) Assume first that there exists $x_0 \in \Omega$ such that $u_{a^\ast}(x_0) = 0$. By definition of $a^\ast$,
\[
u_{a^\ast} \geq 0 \quad \text{in } \Omega.
\]
Using the representation \eqref{chara} of $u_{a^\ast}$, this yields $u_f - \lambda(a^\ast, f, \Omega) e \geq 0$ in $\Omega$, so
\[
\frac{u_f}{e} \geq \lambda(a^\ast, f, \Omega) \quad \text{in } \Omega.
\]
Taking the infimum over $\Omega$ and using \eqref{star}, we obtain
\begin{equation}\label{ok2}
\lambda^\ast(f) \geq \lambda(a^\ast, f, \Omega).
\end{equation}
On the other hand, since $u_{a^\ast}(x_0) = 0$, we have
\begin{equation}\label{ok3}
\lambda(a^\ast, f, \Omega) = \frac{u_f(x_0)}{e(x_0)} \geq \lambda^\ast(f).
\end{equation}
Combining \eqref{ok2} and \eqref{ok3} gives
\begin{equation}\label{ok4}
\lambda^\ast(f) = \lambda(a^\ast, f, \Omega),
\end{equation}
and by \eqref{ok3}, the infimum in \eqref{star} is attained at $x_0$.

\medskip
Conversely, assume that $\lambda^\ast(f)$ is attained at some $x_0 \in \Omega$. By Lemma~\ref{lem:lambda-equality}(ii) and Remark~\ref{rmq:lambda}, the parameter
\begin{equation}\label{isso}
\overline{a} := \frac{\lambda^\ast(f)}{\int_\Omega f e \, dy - \lambda^\ast(f) \int_\Omega e \, dy}
\end{equation}
is well-defined and positive, satisfying
\begin{equation}\label{agora}
\lambda(\overline{a}, f, \Omega) = \lambda^\ast(f).
\end{equation}
By \eqref{chara}, we derive
\begin{equation}\label{nisso}
u_{\overline{a}} = u_f - \lambda(\overline{a}, f, \Omega) e \geq 0 \quad \text{in } \Omega,
\end{equation}
where the inequality follows from \eqref{agora} and the definition \eqref{star}. Since $\lambda^\ast(f)$ is attained at $x_0$, we have $\lambda^\ast(f) = u_f(x_0) / e(x_0)$, which implies
\begin{equation}\label{disso}
u_{\overline{a}}(x_0) = 0.
\end{equation}
Furthermore, since $u_{\overline{a}} \geq 0$ in $\Omega$ and $u_{\overline{a}} = 0$ on $\partial\Omega$, we have
\begin{equation}\label{acabamos}
\partial_\eta u_{\overline{a}} \leq 0 \quad \text{on } \partial\Omega.
\end{equation}
Relations \eqref{nisso}, \eqref{disso}, and \eqref{acabamos} describe interior contact at $\overline a$.
If $0\leq b<\overline a$, then Lemma~\ref{smart} gives $\lambda(b,f,\Omega)<\lambda^\ast(f)$, so
\[
u_b=u_{\overline a}+\bigl(\lambda^\ast(f)-\lambda(b,f,\Omega)\bigr)e>0
\quad\text{in }\Omega,
\]
and
\[
\partial_\eta u_b
=
\partial_\eta u_{\overline a}
+\bigl(\lambda^\ast(f)-\lambda(b,f,\Omega)\bigr)\partial_\eta e<0
\quad\text{on }\partial\Omega.
\]
On the other hand, if $b>\overline a$, then
\[
u_b(x_0)=e(x_0)\bigl(\lambda^\ast(f)-\lambda(b,f,\Omega)\bigr)<0.
\]
Therefore, $\overline a=a^\ast(f)$, which implies $u_{a^\ast}(x_0)=0$.
This identification agrees with \cite[Theorem $2.5$]{BahiaPimentaSantos}.
\end{proof}


As an immediate consequence of Theorem~\ref{teo:interior}, we obtain an explicit formula for the critical parameter $a^*(f)$:
\begin{corollary}\label{cor:a-star-formula}
Let $f \in \mathcal{A}$. Then, the critical parameter $a^*(f)$ satisfies the explicit formula
\begin{equation}\label{formula1}
a^*(f) = \frac{\lambda^\ast(f)}{\displaystyle\int_\Omega f e \, dy - \lambda^\ast(f) \int_\Omega e \, dy}.
\end{equation}
\end{corollary}

\begin{proof}
According to \cite[item (\textit{ii}) of Theorem~2.5]{BahiaPimentaSantos}, at the critical threshold, only two possibilities can occur: either there exists $x_1 \in \partial\Omega$ such that $\partial_\eta u_{a^*}(x_1) = 0$, or there exists $x_0 \in \Omega$ such that $u_{a^*}(x_0) = 0$. 
{\color{black}If the first case holds, we have statement~(\textit{i}) of Theorem~\ref{teo:interior}, then $\lambda^\ast(f) = \lambda^{\ast\ast}(f)$ which is equal to $ \lambda(a^\ast, f, \Omega)$  thanks to \eqref{iphone}. Using  \eqref{pulo}, yields \eqref{formula1}. If the second case holds, we have statement~(\textit{ii}) of Theorem~\ref{teo:interior},} the conclusion follows directly from \eqref{ok4} together with \eqref{pulo}.
\end{proof}

This result is particularly significant because it provides an explicit formula to compute the critical parameter $a^*(f)$ directly from $f$ in any case of breakdown of the maximum in the critical threshold. Such a characterization was not previously available in the literature, and this formula will play a central role in Section~\ref{otimiz}.

Furthermore, Theorem~\ref{teo:interior} allows us to fully classify the boundary and interior behavior of $u_{a^*}$ as follows.

\begin{corollary}\label{corol:main}
Suppose that $f \in \mathcal{A}$. Then the following statements hold:
\begin{itemize}
    \item [(i)] If $\lambda^\ast(f)$ is attained in $\Omega$ and $\lambda^\ast(f) = \lambda^{\ast\ast}(f)$, then $u_{a^\ast}$ vanishes at some point in $\Omega$, and $\partial_\eta u_{a^\ast}$ vanishes at some point on $\partial\Omega$.
    \item [(ii)] If $\lambda^\ast(f)$ is attained in $\Omega$ and $\lambda^\ast(f) < \lambda^{\ast\ast}(f)$, then $u_{a^\ast}$ vanishes at some point in $\Omega$, and $\partial_\eta u_{a^\ast} < 0$ on $\partial\Omega$.
    \item [(iii)] If $\lambda^\ast(f)$ is not attained in $\Omega$ and $\lambda^\ast(f) = \lambda^{\ast\ast}(f)$, then $u_{a^\ast} > 0$ in $\Omega$, and $\partial_\eta u_{a^\ast}$ vanishes at some point on $\partial\Omega$.
\end{itemize}
\end{corollary}
\begin{proof}
The three items are immediate consequences of Lemma \ref{smart} and Theorem \ref{teo:interior}. 
%
%
\end{proof}
\begin{remark}\label{rmq:simultaneous}
Note that the two conditions ``$\lambda^\ast(f)$ is not attained in $\Omega$'' and ``$\lambda^\ast(f) < \lambda^{\ast\ast}(f)$'' can never occur simultaneously. Indeed, by \cite[Theorem~2.5(ii)]{BahiaPimentaSantos}, the breakdown of the maximum principle at $a = a^\ast$ requires that either $u_{a^\ast}$ vanishes at some interior point of $\Omega$ or its normal derivative $\partial_\eta u_{a^\ast}$ vanishes at some boundary point of $\partial\Omega$. By Theorem~\ref{teo:interior}, this implies that either $\lambda^\ast(f)$ is attained in $\Omega$ or $\lambda^\ast(f) = \lambda^{\ast\ast}(f)$.
\end{remark}

The following result is an immediate consequence of item $(ii)$ of Theorem \ref{teo:interior} and clarifies the question raised in Section 5.1 of \cite{BahiaPimentaSantos} regarding assumptions on $f$ that guarantee the occurrence of dead zones.

\begin{corollary}\label{suficiente}
Suppose that $f\in\mathcal{A}$. If there exists an open set $\Omega_0\subset\Omega$ such that $\lambda^\ast(f)$ is attained for all $x\in\Omega_0$, that is,
\begin{equation}\label{zonamorta}
\lambda^\ast(f)=\frac{u_f(x)}{e(x)}, \ \mbox{for all} \ x\in\Omega_0,
\end{equation}
then $\Omega_0$ is a dead zone of $u_{a^\ast}$.    
\end{corollary}

\begin{proof}
By Theorem \ref{teo:interior} and \eqref{chara},
\[
u_{a^\ast}=u_f-\lambda^\ast(f)e.
\]
Hence \eqref{zonamorta} gives $u_{a^\ast}(x)=0$ for every $x\in\Omega_0$.
\end{proof}

The authors of \cite{BahiaPimentaSantos} showed that if $\Omega_0$ is a dead zone of $u_{a^\ast}$, then $f$ is constant in $\Omega_0$ (see Lemma 5.1). However, this is not a sufficient condition for the existence of dead zones (see the example in Subsection 5.1 of \cite{BahiaPimentaSantos}). Notice that assumption \eqref{zonamorta} is stronger than simply assuming that $f$ is constant, and it implies
$$
f(x)=\lambda^\ast(f), \ \mbox{for all} \ x\in\Omega_0.
$$

\section{Examples and Applications}\label{exam}

A further consequence of Theorem \ref{teo:interior} and Corollary \ref{corol:main} is that they allow us to determine precisely which type of degeneracy occurs at the critical level, whether it takes place on the boundary or within $\Omega$. This criterion considerably simplifies the construction of concrete examples illustrating each possible scenario of lack of the maximum principle. In what follows, we shall present several such examples as direct applications of Theorem \ref{teo:interior} and Corollary \ref{corol:main}.

\medskip

\subsection{One-dimensional case}

Consider the interval $\Omega = (0,1)$. Then, we have that
$$
e(x) = \frac{1}{2}x(1-x)
$$
and 
$$
e_x(0) = \frac{1}{2} \quad \hbox{and}\quad   e_x(1)=-{1\over2}.
$$

\subsubsection{Example 1: Lack of maximum principle in the interior and on the boundary}
Consider $f(x) = \sin^2(2\pi x)$. The solution $u_f$ is given by:
\[
    u_f(x) 
    = \frac{1}{2}e(x) + \frac{1}{16\pi^2}\sin^2(2\pi x).
\]
The ratio $\Phi(x) := \frac{u_f(x)}{e(x)}$, for $x\in \Omega$, is given by:
\[
    \Phi(x) 
    = \frac{1}{2} + \frac{\sin^2(2\pi x)}{8\pi^2 x(1-x)}.
\]

Since $\sin^2(2\pi x) \ge 0$ and $x(1-x) > 0$ for $x \in (0,1)$, it follows that $\Phi(x) \geq \frac{1}{2}$. At the interior point $x = \frac{1}{2}$, the sine term vanishes, yielding:
\[
    \Phi\left(\frac{1}{2}\right) 
    = \frac{1}{2}.
\]
Hence, $x = \frac{1}{2}$ is an interior minimum point.

To investigate the behavior at the boundary, 
we obtain:
\[
    \lim_{x \to 0^+} \Phi(x)
    = \frac{1}{2}
\quad \hbox{and}\quad
    \lim_{x \to 1^-} \Phi(x) 
    = \frac{1}{2}.
\]
This confirms that the minimum value $\lambda^*(f) = \frac{1}{2}$ is attained both at the interior and at the boundary of the domain. In particular, we have $\lambda^*(f)=\lambda^{\ast\ast}(f)$. 

By  item $(i)$ of Corollary \ref{corol:main},
we get that $u_{a^\ast}$ vanishes at some point of $\Omega$ and $\partial_{\eta}u_{a^\ast}$ vanishes at some point of $\partial\Omega$. Indeed, we have that 
$$
    u_{a^*}(x)=e(x)\left(\Phi(x) - \frac{1}{2}\right)
$$
and
$$
    u_{a^*,x}(x)=e(x)\Phi_x(x)+e_x(x)\left(\Phi(x) - \frac{1}{2}\right).
$$
Then, we have
$ u_{a^*}(1/2)=0$
and     
$
    \partial_\eta u_{a^*}(0)=0
$
and 
$
    \partial_\eta u_{a^*}(1)=0.
$

This is an example where the lack of maximum principle happens in the interior and the Hopf Lemmma does not hold on the boundary (Figure \ref{FIG1}).
 
\begin{figure}[h]
    \centering
    \includegraphics[scale=0.35
    ]{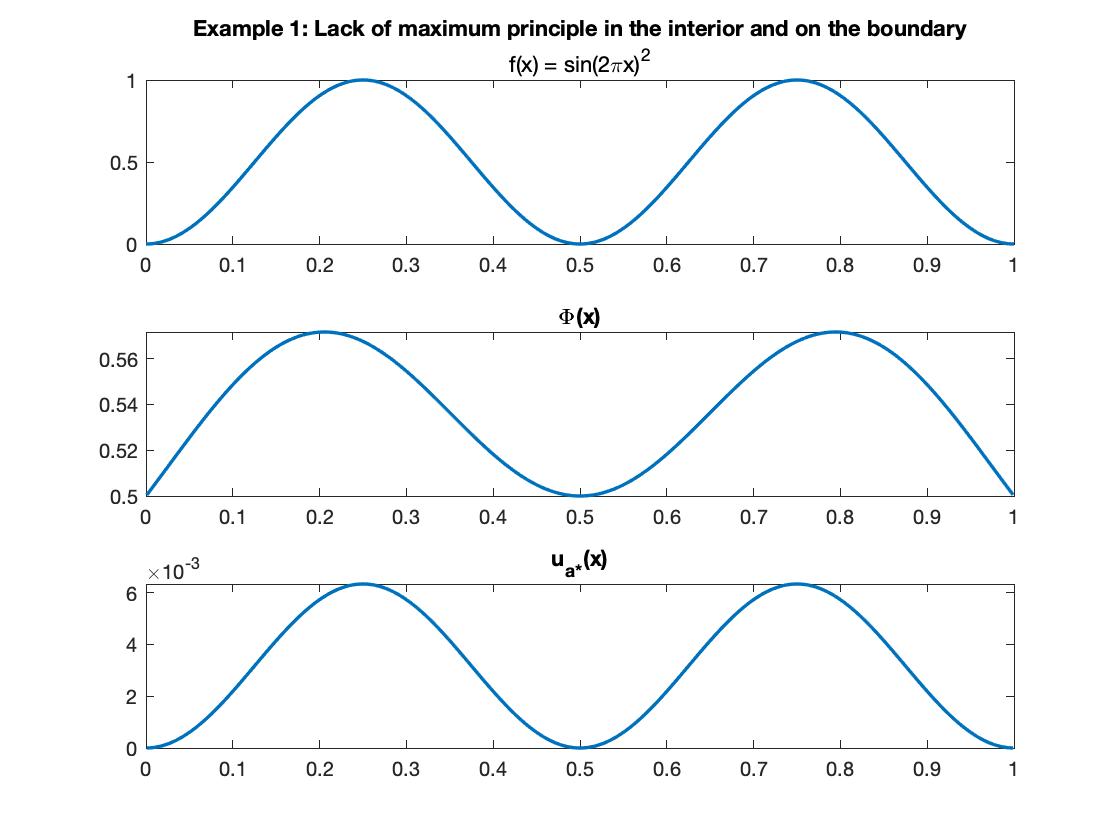}
    \caption{Behavior of $f(x)$, $\Phi(x)$, and $u_{a^*}(x)$ in Example 1.}\label{FIG1}
\end{figure}

\subsubsection{Example 2: Lack of maximum principle  on the boundary}
Let $f(x) = x(1-x)$. The solution $u_f$ is given by:
\begin{equation}
    u_f(x) = \frac{1}{6}e(x)[1+2e(x)]
\end{equation}
The ratio $\Phi(x) = \frac{u_f(x)}{e(x)}$ simplifies to:
\begin{equation}
    \Phi(x) = \frac{1}{6}[1+2e(x)].
\end{equation}
Note that $\Phi'(x) = \frac{1}{6}(1-2x)$. The unique critical point $x={1\over2}$ is a global maximum. Therefore, the minimum of the ratio is attained at the boundaries:
\begin{equation}
    \lambda^{\ast}(f) =\lambda^{\ast\ast}(f)= \frac{1}{6}.
\end{equation}
This confirms that the infimum value $\lambda^*(f) = \frac{1}{6}$ is attained only at the boundary of the domain and not in the interior of the domain. 
By  item $(iii)$ of Corollary \ref{corol:main},
we get that the solution $u_{a^\ast}$ is positive in $\Omega$ and $\partial_{\eta}u_{a^\ast}$ vanishes at some point of $\partial\Omega$. Indeed, we have that 
$$
    u_{a^*}(x)=e(x)\left(\Phi(x) - \frac{1}{6}\right)={1\over3}e^2(x)\quad
\hbox{and}\quad 
    u_{a^*,x}(x)={2\over3}e(x)e_x(x).
$$
Therefore, we obtain $u_{a^*}(x)>0$, for all $x\in(0,1)$, and $    \partial_\eta u_{a^*}(0)=
    \partial_\eta u_{a^*}(1)=0$.

In this example the lack of maximum principle occurs only on the boundary (Figure \ref{FIG2}).

\begin{figure}[h]
    \centering\includegraphics[scale=0.35]{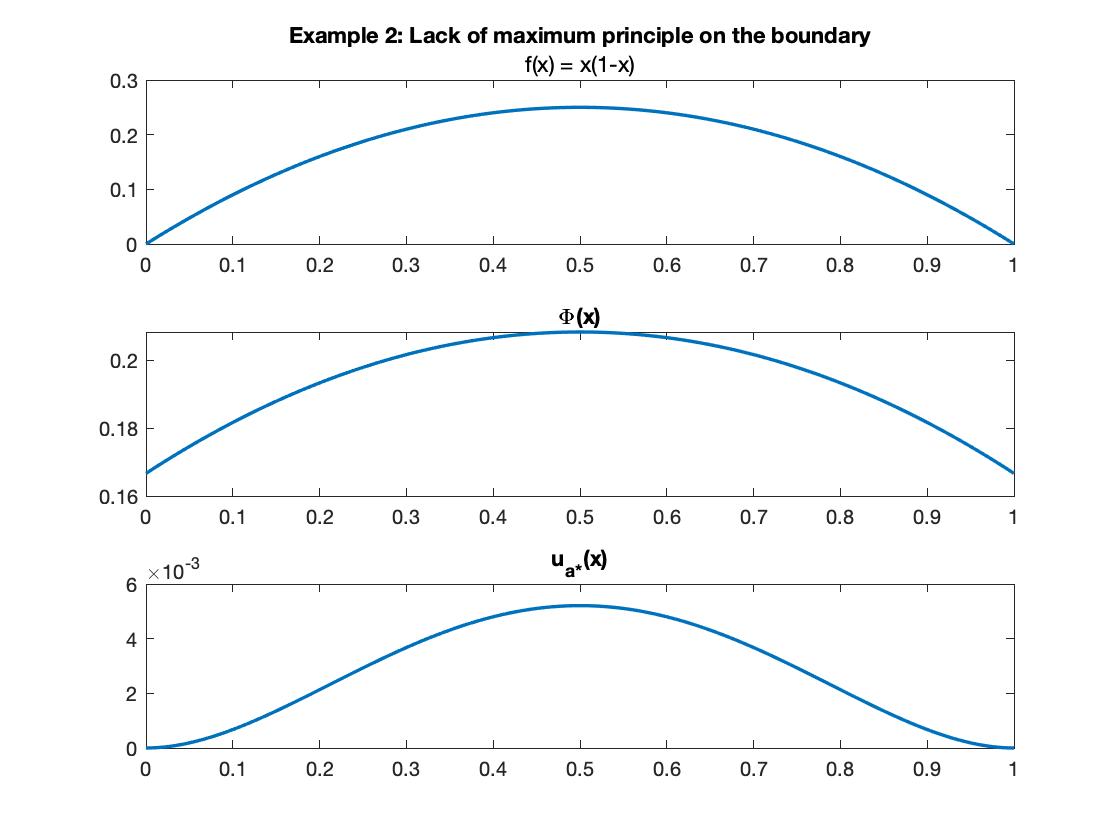}
    \caption{Behavior of $f(x)$, $\Phi(x)$, and $u_{a^*}(x)$ in Example 2.}\label{FIG2}
\end{figure}

\subsubsection{Example 3: Lack of maximum principle  in the interior}
   Consider the nonnegative source term $f(x) = \left(\frac{1}{2}-x\right)^2$. Then, we have that
$$
    u_f(x)=\frac{e(x)[1-4e(x)]}{12}
$$
and we obtain the ratio $ \displaystyle\Phi(x)={u_f(x)\over e(x)}$ is given by
\begin{equation}
    \Phi(x) = \frac{1}{12}[1-4e(x)], \quad \forall x \in (0,1).
\end{equation}
Thus, the minimum of $ \Phi$ is attained at $x = \frac{1}{2}$, where $\Phi(1/2) = \frac{1}{24}$. Then, $\lambda^*(f)$ is attained at $x=1/2$ and $\lambda^*(f)=\frac{1}{24}$
and, since  $\Phi(0)=\Phi(1)=\frac{1}{12}$, we have that $\lambda^{**}(f)=\frac{1}{12}$.
Therefore, since $\lambda^*(f)<\lambda^{**}(f)$, by item $(ii)$ of Corollary \ref{corol:main}, we have that $u_{a^\ast}$ vanishes at some point in $\Omega$, and $\partial_\eta u_{a^\ast} < 0$ on $\partial\Omega$. Indeed, we have that
$$
    u_{a^*}(x)=e(x)\left(\Phi(x) - \frac{1}{24}\right)
\quad\hbox{and}\quad
    u_{a^*,x}(x)=e(x)\Phi_x(x)+e_x(x)\left(\Phi(x) - \frac{1}{24}\right).
$$
Therefore, $ u_{a^*}(1/2)=0$ and
$$
    u_{a^*,x}(0)={1\over48} \quad \hbox{and}\quad   u_{a^*,x}(1)=-{1\over48}.
$$

In this example, the lack of maximum principle happens only in the interior (Figure \ref{FIG3}). 

\begin{figure}[h]
    \centering\includegraphics[scale=0.35]{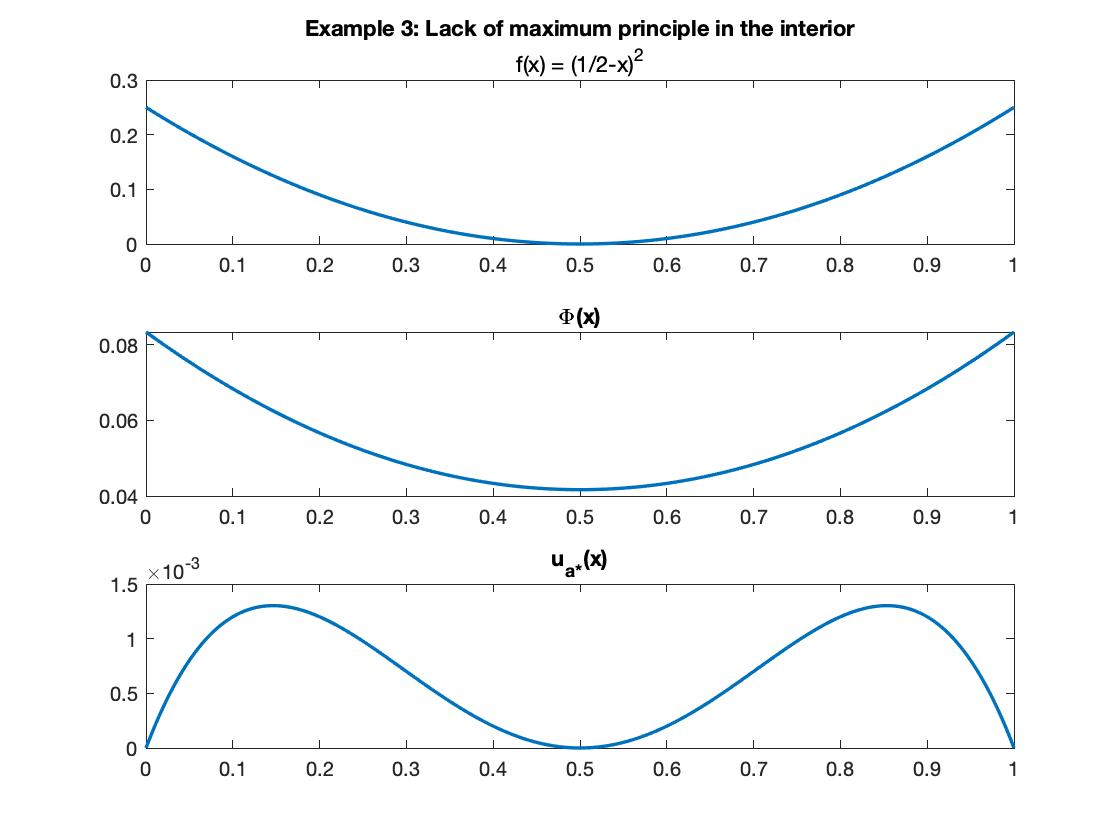}
    \caption{Behavior of $f(x)$, $\Phi(x)$, and $u_{a^*}(x)$ in Example 3.}\label{FIG3}
\end{figure}

\subsubsection{Example 4: Lack of maximum principle with an open dead zone in the interior}
Let $f$ be the piecewise continuous, with two jump discontinuities, positive source term defined by:
\begin{equation}
f(x) = \begin{cases} \frac{11}{9} - x & x \in [0, 1/3) \\ 1 & x \in [1/3, 2/3] \\ x + \frac{2}{9} & x \in (2/3, 1]. \end{cases}
\end{equation}
The solution $u_f$ to the Dirichlet problem is:
\begin{equation}
u_f(x) = \begin{cases} \frac{1}{6}x^3 - \frac{11}{18}x^2 + \frac{14}{27}x & x \in [0, 1/3], \\ e(x) & x \in [1/3, 2/3], \\ \frac{1}{6}(1-x)^3 - \frac{11}{18}(1-x)^2 + \frac{14}{27}(1-x) & x \in [2/3, 1]. \end{cases}
\end{equation} 

It is not difficult to show that $u_f \in H_0^1(0,1)$ and 
%
that $u_f$ is not a classical solution of the problem but a strong solution. In fact, we can prove that $u_f\in W^{2,\infty}(0,1)$.

Let us now analyze the ratio $\Phi(x)={u_f(x)\over e(x)}$. First of all, we have that $\Phi(x) = 1$ for all $x \in [1/3, 2/3]$. On the other hand, on the interval $(0, 1/3)$, we define the difference function $g(x)$:
\begin{equation}
g(x) = u_f(x) - e(x) =
\frac{x}{54} (3x - 1)^2.
\end{equation}
Since $x > 0$ and $(3x-1)^2 > 0$ for all $x \in (0, 1/3)$, it follows that:
\[
g(x) > 0
\]
and then
\[
\Phi(x) > 1,\quad \forall x\in (0,1/3).
\]

By symmetry, we also have
$\Phi(x) > 1,\quad \forall x\in (2/3,1).$

Finally, at the boundaries, the limits are:
\begin{equation}
    \lim_{x \to 0^+} \Phi(x) = \lim_{x \to 1^-} \Phi(x) = \frac{28}{27} > 1.
\end{equation}

Thus, the global minimum $\lambda^*(f) = 1$ is attained on the entire interior sub-interval $[1/3, 2/3]$ and, since  $\Phi(0)=\Phi(1)=\frac{28}{27}$, we have that $\lambda^{**}(f)=\frac{28}{27}$.
Therefore,  $\lambda^{\ast}(f)<\lambda^{**}(f)$ and,  by item $(ii)$ of Corollary \ref{corol:main}, we have that $u_{a^\ast}$ vanishes at some point of $\Omega$ and $\partial_{\eta}u_{a^\ast}$ is negative on $\partial\Omega$. Indeed, we have that
$$
    u_{a^*}(x)=e(x)\left(\Phi(x) - 1\right)
$$
and then
$$
    u_{a^*,x}(x)=e(x)\Phi_x(x)+e_x(x)\left(\Phi(x) - 1\right).
$$
Therefore, $ u_{a^*}(x)=0$ for any $x\in [1/3,2/3]$ and
$$
    u_{a^*,x}(0)={1\over54} \quad \hbox{and}\quad   u_{a^*,x}(1)=-{1\over54}.
$$

This is also an example for Corollary \ref{suficiente}, where the lack of maximum principle happens through the emergence of a dead zone (Figure \ref{FIG4}). 

\begin{figure}[h]\centering\includegraphics[scale=0.35]{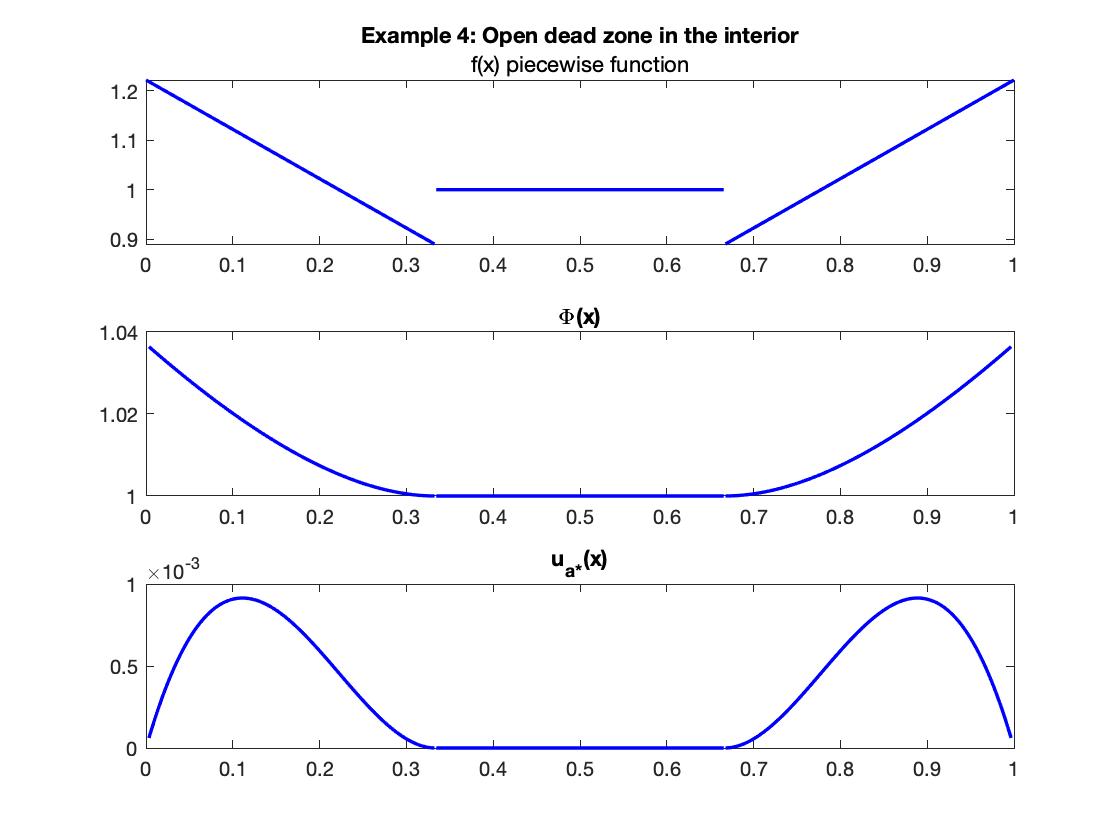}
    \caption{Behavior of $f(x)$, $\Phi(x)$, and $u_{a^*}(x)$ in Example 4.}\label{FIG4}
\end{figure}

\subsection{Two-dimentional case}

We consider the domain $\Omega = B_1(0) \subset \mathbb{R}^2$. The Laplace operator in polar coordinates is: $\Delta_{r,\theta}= \partial_r^2 + {1\over r}\partial_r+{1\over r^2}\partial_\theta^2$. Let $e(\mathbf{x})$ be the solution to the torsion problem 
$$
    -\Delta e = 1 \quad \text{in } \Omega,\qquad e = 0 \quad \text{on } \partial\Omega,
$$
 which is  radial and given by:
\[
e(r) = \frac{1 - r^2}{4}, \quad \text{where } r = |\mathbf{x}|
\]
and
\[
\nabla e(\mathbf{x}) = -{1\over2}\mathbf{x}.
\]

We define a strictly positive, bounded, and radially symmetric source term $f(r)$ defined by:
\begin{equation}
f(r) = \begin{cases} 
1 & r \in [0, 1/2) \\[0.5em]
\frac{9}{2}r - 3 + \frac{5}{8r} & r \in [1/2, 1]. 
\end{cases}
\end{equation}
The solution $u_f$ to the Dirichlet problem $-\Delta u = f$ in $\Omega$ with $u=0$ on $\partial\Omega$ is:
\begin{equation}
u_f(r) = \begin{cases} 
e(r) & r \in [0, 1/2], \\[0.5em]
-\frac{1}{2}r^3 + \frac{3}{4}r^2 - \frac{5}{8}r + \frac{3}{8} & r \in (1/2, 1]. 
\end{cases}
\end{equation} 

It is not difficult to prove that $u_f \in H_0^1(\Omega)$, 
%
%
In this case, $u_f$ is not a classical solution of the problem but a strong solution.

Let us now analyze the ratio $\Phi(\mathbf{x}) = \frac{u_f(\mathbf{x})}{e(\mathbf{x})}$. First of all, we have that $\Phi(r) = 1$ for all $r \in [0, 1/2]$. 
On the other hand, we define the difference function $g(r)$:
\begin{equation}
g(r) = u_f(r) - e(r) 
= \frac{1}{2}\left(r - \frac{1}{2}\right)^2 (1 - r).
\end{equation}
Since $r < 1$ and $(r - 1/2)^2 > 0$ for all $r \in (1/2, 1)$, it follows that:
\[
g(r) > 0 
\]
and then
$$
\Phi(r) > 1, \quad \forall r \in (1/2, 1).
$$

Finally, at the boundary $\partial\Omega$ ($\{r = 1\}$), the limit is:
\begin{equation}
    \lim_{r \to 1^-} \Phi(r) = \frac{u_f'(1)}{e'(1)} = \frac{-5/8}{-1/2} = \frac{5}{4} > 1.
\end{equation}

Thus, the global minimum $\lambda^*(f) = 1$ is attained on the entire interior closed ball $\overline{B_{1/2}(0)}$ and, since $\Phi|_{\partial\Omega} = \frac{5}{4}$, we have that $\lambda^{**}(f) = \frac{5}{4}$.
Therefore, $\lambda^{\ast}(f) < \lambda^{**}(f)$ and, by item $(ii)$ of Corollary \ref{corol:main}, we have that $u_{a^\ast}$ vanishes at some point of $\Omega$ and $\partial_{\eta}u_{a^\ast}$ is negative on $\partial\Omega$. Indeed, we have that
\[
    u_{a^*}(\mathbf{x}) = e(\mathbf{x})\left(\Phi(\mathbf{x}) - 1\right)
\]
and
$$
    \nabla u_{a^*}(\mathbf{x}) = e(\mathbf{x})\nabla \Phi(\mathbf{x})+\nabla e(\mathbf{x})\left(\Phi(\mathbf{x}) - 1\right).
$$
 Therefore, $u_{a^*}(\mathbf{x}) = 0$ for any $\mathbf{x} \in \overline{B_{1/2}(0)}$, while on the boundary $\partial\Omega$:
\[
    \partial_{\eta} u_{a^*}\Big|_{\partial\Omega}= -\frac{1}{8}.
\]
Here, notice that the outward unit normal vector at a boundary point $x\in \partial\Omega$ is  $\eta(x)=x$.

This is a two-dimensional example for Corollary \ref{suficiente}, where the lack of maximum principle happens through the emergence of an open dead zone $\overline{B_{1/2}(0)}$, (Figure \ref{FIG5}).
\begin{figure}[h]\centering\includegraphics[scale=0.2]{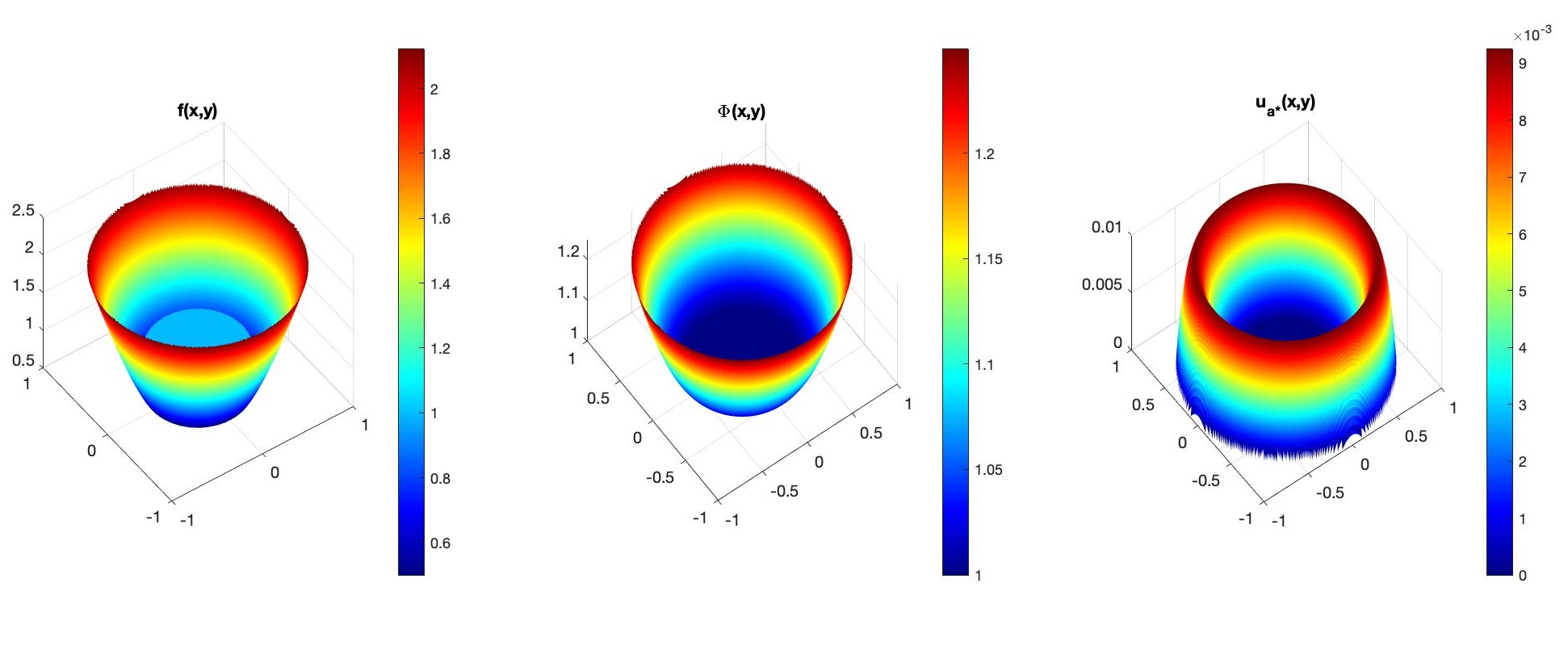}
    \caption{Behavior of $f(x)$, $\Phi(x)$, and $u_{a^*}(x)$ in two-dimentional case.}\label{FIG5}
\end{figure}

\section[An optimization problem involving critical parameter]{An optimization problem involving \texorpdfstring{$a^*$}{a*}}\label{otimiz}

In a biological context, the presence of dead-core solutions to the problem
\begin{equation}\label{eq:P}
\left\{
\begin{array}{lll}
-\Delta u= f-a^{\ast}(f) \displaystyle\int_\Omega u_{a^\ast} \,dy, &in& \Omega,\\[1.2em]
u = 0, & on& \partial\Omega,
\end{array}
\right.
\end{equation}
represents the formation of a hostile region inside the habitat. Within a species-preservation framework, it is therefore natural to seek resource profiles that maximize the threshold at which such a region first appears. More precisely, we ask whether the critical level $a^\ast(f)$ admits a maximizer in suitable admissible classes. A larger critical level means that the nonlocal crowding intensity can increase further before strict positivity is lost. This optimization problem motivates the results of this section.

\begin{definition}
We say that a function $f \in \mathcal{A}$ satisfies the interior-contact condition if $u_{a^\ast(f)}(x_f)=0$ for some $x_f\in\Omega$.
\end{definition}

In what follows, we fix $M>0$, $0<C<M|\Omega|$, and $\delta>0$ such that $\Omega_\delta$ is nonempty, and define
\[
\Omega_\delta
:=
\left\{
x\in\Omega:
\operatorname{dist}(x,\partial\Omega)\ge\delta
\right\}.
\]

\subsection{Prescribed total mass}

Under these assumptions, a natural admissible class of resource distributions with prescribed total mass is
\[
\mathcal A_{\delta,C,M}
:=
\left\{
f\in\mathcal A:
\begin{array}{l}
0\le f\le M \quad \text{a.e. in }\Omega,\\[2mm]
\displaystyle\int_\Omega f(y)\,dy=C,\\[2mm]
\exists\,x_f\in\Omega_\delta
\text{ such that }
u_{a^*}(x_f)=0
\end{array}
\right\}.
\]

The upper bound models limited resources, while the last condition ensures, by Theorem \ref{teo:interior}, that contact occurs at a point separated from the boundary.

The following result shows that the fixed-mass constraint, even together with the uniform bound and the interior-contact condition, does not provide a uniform upper bound for the critical threshold:
\[
\sup_{f\in\mathcal A_{\delta,C,M}} a^*(f)=+\infty.
\]
\begin{theorem}\label{ttteo}
Let $c:=C/|\Omega|$. There exists a sequence $\{f_n\}\subset\mathcal A_{\delta,C,M}$ such that
$$
f_n\longrightarrow c
\ \ \text{in} \ L^\infty(\Omega),
$$
$$
\int_{\Omega}u_{a^\ast(f_n)}(y) dy\longrightarrow 0
$$
and
$$
a^\ast(f_n)\longrightarrow+\infty.
$$
\end{theorem}

\begin{proof}
Fix
$$
x_0\in\Omega_\delta.
$$
We shall construct a function $\varphi$ satisfying
$$
\int_\Omega\varphi\,dy=0
$$
such that the corresponding solution $u_\varphi$ of
$$
\begin{cases}
-\Delta u_\varphi=\varphi & \text{in }\Omega,\\
u_\varphi=0 & \text{on }\partial\Omega
\end{cases}
$$
has the property that
$$
\frac{u_\varphi}{e}
$$
has a strict global minimum at $x_0$. To do this, set
$$
q(x):=|x-x_0|^2.
$$
We look for $u_\varphi$ in the form
$$
u_\varphi(x)=e(x)\big(q(x)-\alpha\big),
$$
where $\alpha\in\mathbb R$ will be chosen so that
$$
\int_\Omega\varphi\,dy=0.
$$

Clearly,
$$
u_\varphi=0
\qquad\text{on }\partial\Omega,
$$
because $e=0$ on $\partial\Omega$. Moreover,
$$
\frac{u_\varphi(x)}{e(x)}
=
q(x)-\alpha
\qquad\text{for every }x\in\Omega.
$$
Consequently,
$$
\frac{u_\varphi(x_0)}{e(x_0)}
=-\alpha,
$$
whereas, for every $x\neq x_0$,
$$
\frac{u_\varphi(x)}{e(x)}
=
q(x)-\alpha
>
-\alpha
=
\frac{u_\varphi(x_0)}{e(x_0)}.
$$
Thus $x_0$ is the unique global minimizer of $u_\varphi/e$.

It remains to choose $\alpha$ so that
$$
\int_\Omega\varphi\,dy=0.
$$

Define
$$
\varphi:=-\Delta u_\varphi.
$$
Since $\Omega$ is smooth and $e$ is the torsion function, $\varphi\in L^\infty(\Omega)$.
By the divergence theorem,
$$
\int_\Omega\varphi\,dy
=
-\int_\Omega\Delta u_\varphi\,dy
=
-\int_{\partial\Omega}\partial_\nu u_\varphi\,dS.
$$
Since $e=0$ on $\partial\Omega$,
$$
\partial_\nu u_\varphi
=
(q-\alpha)\partial_\nu e
\qquad\text{on }\partial\Omega.
$$
Therefore,
$$
\int_\Omega\varphi\,dy
=
-\int_{\partial\Omega}
(q-\alpha)\partial_\nu e\,dS.
$$

We choose
$$
\alpha
:=
\frac{
\displaystyle\int_{\partial\Omega}
q(x)(-\partial_\nu e)\,dS
}{
\displaystyle\int_{\partial\Omega}
(-\partial_\nu e)\,dS
}.
$$
Since
$$
-\partial_\nu e>0
\qquad\text{on }\partial\Omega,
$$
the denominator is strictly positive. With this choice,
$$
\begin{aligned}
\int_\Omega\varphi\,dy
&=
\int_{\partial\Omega}
\big(q(x)-\alpha\big)(-\partial_\nu e)\,dS =0.
\end{aligned}
$$
Hence
$$
\int_\Omega\varphi\,dy=0.
$$

Notice also that $\varphi$ is not constant. Indeed, if $\varphi=K$, then $u_\varphi$ would solve
$$
-\Delta u_\varphi=K
\quad\text{in }\Omega,
\qquad
u_\varphi=0
\quad\text{on }\partial\Omega,
$$
and therefore $u_\varphi=Ke$. Since $e>0$ in $\Omega$, this would imply
$$
q(x)-\alpha=K
\qquad\text{in }\Omega,
$$
or equivalently
$$
q(x)=K+\alpha
\qquad\text{in }\Omega,
$$
which is impossible. Hence $\varphi$ is nonconstant.

Now define
$$
f_n:=c+\frac1n\varphi.
$$

Since
$$
\int_\Omega\varphi\,dy=0,
$$
we have
$$
\int_\Omega f_n\,dy
=
c|\Omega|+\frac1n\int_\Omega\varphi\,dy
=
C.
$$

Moreover,
$$
\|f_n-c\|_{L^\infty(\Omega)}
=
\frac1n\|\varphi\|_{L^\infty(\Omega)}
\longrightarrow0.
$$
Thus
$$
f_n\longrightarrow c
\qquad\text{strongly in }L^\infty(\Omega).
$$

Since $M>c$, choose $n_0$ sufficiently large such that
$$
\frac1n\|\varphi\|_{L^\infty(\Omega)}
\leq
\min\{c,M-c\}
\qquad\forall n\geq n_0.
$$
Then, for $n\geq n_0$,
$$
0\leq f_n\leq M
\qquad\text{in }\Omega.
$$
Since $\varphi$ is nonconstant, $f_n$ is nonconstant. Hence
$$
f_n\in\mathcal A
\qquad\forall n\geq n_0.
$$

It remains to verify the interior-contact condition.

Let $u_{f_n}$ denote the solution of
$$
\begin{cases}
-\Delta u_{f_n}=f_n & \text{in }\Omega,\\
u_{f_n}=0 & \text{on }\partial\Omega.
\end{cases}
$$
By linearity,
$$
-\Delta\left(ce+\frac1n u_\varphi\right)
=
c+\frac1n\varphi
=
f_n.
$$
Since both $e$ and $u_\varphi$ vanish on $\partial\Omega$, uniqueness gives
$$
u_{f_n}
=
ce+\frac1n u_\varphi.
$$

Consequently,
$$
\frac{u_{f_n}(x)}{e(x)}
=
c+\frac1n\frac{u_\varphi(x)}{e(x)}
=
c+\frac1n\big(q(x)-\alpha\big).
$$

Therefore,
$$
\frac{u_{f_n}(x_0)}{e(x_0)}
=
c-\frac{\alpha}{n},
$$
while, for every $x\neq x_0$,
$$
\frac{u_{f_n}(x)}{e(x)}
=
c-\frac{\alpha}{n}
+\frac{q(x)}{n}
>
c-\frac{\alpha}{n}.
$$
Hence
$$
\lambda^*(f_n)
=
\inf_{x\in\Omega}
\frac{u_{f_n}(x)}{e(x)}
=
\frac{u_{f_n}(x_0)}{e(x_0)}.
$$

Since
$$
x_0\in\Omega_\delta,
$$
Theorem \ref{teo:interior} gives
\[
u_{a^*(f_n)}
=u_{f_n}-\lambda^*(f_n)e
=\frac1n e q.
\]
Thus the interior-contact condition is satisfied, $u_{a^*(f_n)}(x_0)=0$, and
\[
\int_\Omega u_{a^*(f_n)}\,dy
=\frac1n\int_\Omega e(y)|y-x_0|^2\,dy
\longrightarrow0.
\]
Therefore
$$
f_n\in\mathcal A_{\delta,C,M}
\qquad\forall n\geq n_0.
$$

Since
 $$
   \lambda^*(f_n)=c-\frac{\alpha}{n}, 
 $$
formula \eqref{formula1} yields
  $$
    a^*(f_n)
    =\frac{c-\frac{\alpha}{n}}
    {\displaystyle \frac{1}{n}\int_\Omega e(y)|y-x_0|^2\,dy}
    \longrightarrow +\infty.
 $$
This completes the proof.
\end{proof}

\subsection{Prescribed total population}\label{Prtotpop}

The discussion of the previous subsection leads us to consider the maximization problem under a prescribed lower bound for the total population at the critical threshold. To distinguish this constraint from the fixed-mass constraint, define
\[
\mathcal A^{\mathrm{pop}}_{\delta,C,M}
:=
\left\{
f\in \mathcal{A}:
\begin{array}{l}
0\le f\le M \text{ a.e. in }\Omega,\\[2mm]
\displaystyle\int_\Omega u_{a^\ast}(x) dy\geq C ,\\[2mm]
\exists x_f\in\Omega_\delta
\text{ such that }u_{a^\ast}(x_f)=0
\end{array}
\right\}.
\]
Notice that the condition $\int_\Omega u_{a^*} dy \geq C$ ensures that the total population at the critical threshold remains above a minimum viable level.


We first give a condition ensuring that this class is nonempty.
\begin{proposition}\label{prop:nonempty_M1}
Fix $x_0 \in \Omega_\delta$ and define the geometric constant:
\[
C_0(\Omega,\delta,x_0,M)
:=
\frac{M\displaystyle\int_\Omega e(x)|x-x_0|^2\,dx}
{2\|\varphi\|_{L^\infty(\Omega)}}>0,
\]
where $\varphi$ is the zero-mean function constructed in Theorem~\ref{ttteo}. Then, for any $C \in (0,C_0(\Omega,\delta,x_0,M)]$, the admissible class $\mathcal A^{\mathrm{pop}}_{\delta,C,M}$ is nonempty.
\end{proposition}

\begin{proof}
Fix $x_0 \in \Omega_\delta$ and $C \in (0,C_0(\Omega,\delta,x_0,M)]$. Let $q(x) := |x - x_0|^2$ and let $u_\varphi(x) = e(x)(q(x) - \alpha)$ be the unique solution to
\[
\begin{cases} 
-\Delta u_\varphi = \varphi & \text{in } \Omega,\\ 
u_\varphi = 0 & \text{on } \partial\Omega, 
\end{cases}
\]
where $\alpha \in \mathbb{R}$ is chosen such that $\displaystyle\int_\Omega \varphi \,dx = 0$.

Define the scaling parameter $A > 0$ by
\[
A := \frac{C}{\displaystyle\int_\Omega e(x) |x - x_0|^2 \,dx}.
\]
Since $C \le C_0(\Omega,\delta,x_0,M)$, we have the bound
\[
A \le \frac{M}{2\|\varphi\|_{L^\infty(\Omega)}}.
\]

Now, define the source function $f \colon \Omega \to \mathbb{R}$ as
\[
f(x) := A \left( \|\varphi\|_{L^\infty(\Omega)} + \varphi(x) \right).
\]
We claim that $f \in \mathcal A^{\mathrm{pop}}_{\delta,C,M}$ by verifying all required conditions:

\textbf{Step 1: $0 \le f \le M$ a.e. in $\Omega$}\\
    Since $-\|\varphi\|_{L^\infty(\Omega)} \le \varphi(x) \le \|\varphi\|_{L^\infty(\Omega)}$ a.e. in $\Omega$, we immediately have
    \[
    f(x) \ge A \left( \|\varphi\|_{L^\infty(\Omega)} - \|\varphi\|_{L^\infty(\Omega)} \right) = 0,
    \]
    and using the upper bound on $A$,
    \[
    f(x) \le A \left( 2 \|\varphi\|_{L^\infty(\Omega)} \right) \le M.
    \]
    Thus, $0 \le f(x) \le M$ a.e. in $\Omega$.

     \textbf{Step 2: Contact condition in $\Omega_\delta$:}\\
    By linearity of the Laplacian, the solution $u_f$ is:
    \[
    u_f(x) = A \|\varphi\|_{L^\infty(\Omega)} e(x) + A u_\varphi(x) = A e(x) \left( \|\varphi\|_{L^\infty(\Omega)} - \alpha + q(x) \right).
    \]
    Evaluating the ratio $u_f / e$, we obtain
    \[
    \frac{u_f(x)}{e(x)} = A \left( \|\varphi\|_{L^\infty(\Omega)} - \alpha + q(x) \right).
    \]
    Since $q(x) = |x - x_0|^2 \ge 0$ with $q(x_0) = 0$, $x_0 \in \Omega_\delta$ is the unique global minimizer of $u_f/e$, and
    \[
    \lambda^*(f) = \inf_{x\in\Omega} \frac{u_f(x)}{e(x)} = A \left( \|\varphi\|_{L^\infty(\Omega)} - \alpha \right)=\frac{u_f(x_0)}{e(x_0)}.
    \]
    Therefore, the state variable at the critical threshold $a^*(f)$ reads
    \[
    u_{a^*}(x) = u_f(x) - \lambda^*(f) e(x) = A e(x) q(x) = A e(x) |x - x_0|^2.
    \]
    Evaluating at $x_0 \in \Omega_\delta$, we get $u_{a^*}(x_0) = A e(x_0) |x_0 - x_0|^2 = 0$, verifying the contact condition in $\Omega_\delta$.

    \textbf{Step 3: Total population constraint}\\
    Integrating $u_{a^*}(x)$ over $\Omega$ and using the definition of $A$,
    \[
    \int_\Omega u_{a^*}(x) \,dx = A \int_\Omega e(x) |x - x_0|^2 \,dx = \frac{C}{\int_\Omega e(x) |x - x_0|^2 \,dx} \int_\Omega e(x) |x - x_0|^2 \,dx = C.
    \]

Since $f$ is nonconstant (because $\varphi$ is nonconstant), $f \in \mathcal{A}$. Thus $f \in \mathcal A^{\mathrm{pop}}_{\delta,C,M}$, showing that this class is nonempty.
\end{proof}

\begin{remark}
Notice that if $\displaystyle C>M\int_\Omega e\, dx$ then
$A^{\mathrm{pop}}_{\delta,C,M}=\emptyset$.

\end{remark}
The critical parameter is uniformly bounded on the population-constrained class.
\begin{remark}\label{remark_bounded}
Notice that ${a^*(f)} \leq M/C$ for all $f\in \mathcal A^{\mathrm{pop}}_{\delta,C,M}$. Indeed, let $f\in\mathcal A^{\mathrm{pop}}_{\delta,C,M}$. Then
\[
\begin{cases}
-\Delta u_{a^*} + {a^*} \displaystyle\int_\Omega u_{a^*} \,dy = f & \hbox{in}\quad \Omega,\\[1.2em]
u_{a^*} = 0 & \hbox{on}\quad \partial\Omega.
\end{cases}
\]

Multiplying the differential equation by $u_{a^*}$ and integrating by parts, we obtain the energy identity
\[
\displaystyle \int_\Omega|\nabla u_{a^*}|^2 \,dy + {a^*} \displaystyle\left(\int_\Omega u_{a^*} \,dy\right)^2 = \int_\Omega f u_{a^*} \,dy.
\]
Using the fact that $f \leq M$ almost everywhere in $\Omega$, we deduce
\[
\displaystyle \int_\Omega|\nabla u_{a^*}|^2 \,dy + {a^*} \displaystyle\left(\int_\Omega u_{a^*} \,dy\right)^2 \leq M\int_\Omega u_{a^*} \,dy. 
\]
Dropping the nonnegative gradient term and using the total population bound $\displaystyle\int_\Omega u_{a^\ast} \,dy \geq C$, we obtain
\begin{equation}\label{eq:a^*bounded}
{a^*(f)} \leq \frac{M}{\displaystyle\int_\Omega u_{a^*} \,dy} \leq \frac{M}{C}, \quad \forall f\in \mathcal A^{\mathrm{pop}}_{\delta,C,M}.
\end{equation}
\end{remark}

\vspace{0.5cm}

Our main goal in this section is to study the variational problem
\[
\sup_{f\in\mathcal A^{\mathrm{pop}}_{\delta,C,M}}
a^*(f).
\]
Thanks to the uniform upper bound in \eqref{eq:a^*bounded}, the previous supremum is finite. 

\begin{theorem}\label{tteo}
If $\mathcal A^{\mathrm{pop}}_{\delta,C,M}\neq \emptyset $ then the variational problem
\begin{equation}\label{optimala*}
\sup_{f\in\mathcal A^{\mathrm{pop}}_{\delta,C,M}}
a^*(f)
\end{equation}
admits at least one maximizer.
\end{theorem}

\begin{proof} 

Let $\{f_n\}\subset\mathcal A^{\mathrm{pop}}_{\delta,C,M}$ be a maximizing sequence for \eqref{optimala*}, namely,
\begin{equation}\label{eq:conv_a}
a^*(f_n)
\to
\sup_{g\in\mathcal A^{\mathrm{pop}}_{\delta,C,M}}
a^*(g)=:\bar{a}.
\end{equation}
 Thanks to Remark \ref{remark_bounded}, we have that
$\bar{a}$ is a positive real number.

Since
\[
0\le f_n\le M,
\]
the sequence is bounded in $L^\infty(\Omega)$. By the Banach--Alaoglu Theorem, there exists
$f\in L^\infty(\Omega)$ such that, up to a subsequence,
\[
f_n\rightharpoonup^\ast f
\qquad\text{in }L^\infty(\Omega).
\]

The bounds are preserved under weak-$*$ convergence. Indeed, for every
$\varphi\in L^1(\Omega)$ with $\varphi\ge0$,
\[
\int_\Omega f\varphi\,dy
=
\lim_{n\to\infty}
\int_\Omega f_n\varphi\,dy
\ge0.
\]
Applying the same argument to $M-f_n$ gives $f\le M$ a.e. Hence
\begin{equation}\label{maximum}
0\le f\le M.
\end{equation}

Let
\[
u_n:=u_{f_n}.
\]
Then
\[
\begin{cases}
-\Delta(u_n-u_f)=f_n-f & \text{in }\Omega,\\
u_n-u_f=0 & \text{on }\partial\Omega.
\end{cases}
\]

Since $\{f_n\}$ is uniformly bounded in $L^\infty(\Omega)$, elliptic regularity yields, for every $p<+\infty$,
\[
\|u_n\|_{W^{2,p}(\Omega)}
\le C_p.
\]
Choosing $p>N$ and using compact Sobolev embedding, a subsequence converges in $C^1(\overline\Omega)$. The weak formulation and the weak-$*$ convergence of $f_n$ identify the limit with $u_f$; hence
\[
u_n\to u_f
\qquad\text{in }C^1(\overline\Omega),
\]
up to a subsequence. Since $\{f_n\}\subset\mathcal A^{\mathrm{pop}}_{\delta,C,M}$, for each $n$ there exists
$x_n\in\Omega_\delta$
such that
\[
\lambda^*(f_n)
=
\frac{u_n(x_n)}{e(x_n)}.
\]
Since $\Omega_\delta$ is compact,
there exists
$x_0\in\Omega_\delta$
such that, up to a subsequence,
\[
x_n\to x_0.
\]
We claim that
\[
\frac{u_f(x_0)}{e(x_0)}=\lambda^*(f).
\]

Indeed, using the uniform convergence of $u_n$ and the continuity of $e$, we obtain
\[
\frac{u_n}{e}\to \frac{u_f}{e}\quad \hbox{uniformly in }\quad \Omega_\delta.
\]
Consequently,
\[
\lambda^*(f_n)
=
\frac{u_n(x_n)}{e(x_n)}
\to
\frac{u_f(x_0)}{e(x_0)}.
\]

On the other hand, for every fixed $y\in\Omega$,
\[
\frac{u_n(x_n)}{e(x_n)}\le \frac{u_n(y)}{e(y)}.
\]
Passing to the limit, we deduce
\[
\frac{u_f(x_0)}{e(x_0)}\le \frac{u_f(y)}{e(y)}
\qquad\forall y\in\Omega.
\]
Therefore,
\[
\lambda^*(f)
=
\inf_\Omega \frac{u_f}{e}
=
\frac{u_f(x_0)}{e(x_0)},
\]
with $x_0\in\Omega_\delta$ and
\begin{equation}\label{conv:lambda}
\lambda^*(f_n)
\to
\lambda^*(f).
\end{equation}

Next, since $e\in L^1(\Omega)$ and
$
f_n \overset{*}{\rightharpoonup} f$ in $L^\infty(\Omega)$,
we have
\[
\int_\Omega f_n e\,dy
\longrightarrow
\int_\Omega fe\,dy.
\]

By Lemma \ref{ok},
\[
u_{a^*(f_n)}=u_n-\lambda^*(f_n)e.
\]
Thus the preceding convergences imply
\[
u_{a^*(f_n)}\to \bar{u}:=u_f-\lambda^*(f)e
\qquad\text{in }C^1(\overline\Omega),
\]
and therefore
\[
\int_\Omega\bar u\,dy
=\lim_{n\to\infty}\int_\Omega u_{a^*(f_n)}\,dy
\ge C.
\]
In particular, $\bar u\not\equiv0$. If $f$ were constant, then $u_f$ would be a constant multiple of $e$ and $\bar u$ would vanish identically. Hence $f\in\mathcal A$.

By definition of $\lambda^*(f)$, $\bar u\ge0$ in $\overline\Omega$, and the minimizing property of $x_0$ gives $\bar u(x_0)=0$. Consequently, Theorem \ref{teo:interior} identifies $\bar u$ with the critical state $u_{a^*(f)}$. Thus
\[
f\in\mathcal A^{\mathrm{pop}}_{\delta,C,M}.
\]

Finally, \eqref{formula1}, \eqref{conv:lambda}, and the convergence of $\int_\Omega f_ne\,dy$ give
\[
a^*(f_n)\longrightarrow a^*(f).
\]
Together with \eqref{eq:conv_a}, this yields
\[
a^*(f)
=
\sup_{g\in\mathcal A^{\mathrm{pop}}_{\delta,C,M}} a^*(g).
\]

\end{proof}

{\color{black}
\begin{remark}
In dimension one with $\Omega = (0,1)$ and $\delta = 1/2$, $M=2$ and $C={1\over 81}$, a direct rearrangement argument proves that the unique maximizer in $\mathcal A_{\frac12,\frac{1}{81}, 2}^{\mathrm{pop}}$ is the bang-bang function 
$$f^* = 2\chi_{(0,1/3)\cup(2/3,1)}\quad\hbox{and}\quad a^*(f_*)=72.$$
Biologically, the bang–bang structure shows that the optimal strategy concentrates resources at their maximum level near the boundary while leaving a resource-depleted central region. Thus, the habitat is divided into a resource-rich peripheral zone and a central zone where the dead-core first appears.

\end{remark}}

\section{Additional comments}
   
\subsection{Non-Local Elliptic Equations with Vanishing Diffusion}

In this section, we deal with the behavior of solutions to the non-local elliptic equation with respect to a diffusion parameter $\epsilon > 0$:
\begin{equation}\label{eq:eps_problem}
\left\{
\begin{array}{lll}
-\epsilon \Delta u + \displaystyle\int_\Omega u(y) \,dy = f & \text{in}& \Omega,\\[1em]
u = 0 & \text{on}& \partial\Omega.
\end{array}
\right.
\end{equation}

Dividing the differential equation in \eqref{eq:eps_problem} by $\epsilon$ yields the equivalent formulation
$$
-\Delta u + \frac{1}{\epsilon} \int_\Omega u(y) \,dy = \frac{f}{\epsilon} \quad \text{in } \Omega.
$$
This equation naturally fits into our standard framework, corresponding to a non-local parameter $a = \frac{1}{\epsilon}$ and a scaled source term $\frac{f}{\epsilon}$. Thanks to the homogeneity of zero order of the critical parameter (see the item $(iii)$ of Remark \ref{obs}), we have:
$$
a^*\left(\frac{f}{\epsilon}\right) = a^*(f).
$$
This scale invariance allows us to define a critical diffusion threshold for the system:
$$
\epsilon^*(f) := \frac{1}{a^*(f)}.
$$

The solution $u_\epsilon \in C^1(\overline{\Omega})\cap H^2(\Omega)\cap H_0^1(\Omega)$ has the explicit representation
$$
u_{\epsilon}(x) = \frac{1}{\epsilon} \left[ u_f(x) - \lambda(\epsilon, f, \Omega) e(x) \right], \quad \forall x \in \overline{\Omega},
$$
where the non-local multiplier $\lambda(\epsilon, f, \Omega)$ is given by:
$$
\lambda(\epsilon, f, \Omega) := \frac{\displaystyle\int_\Omega f(y) e(y) \,dy}{\displaystyle\epsilon + \int_\Omega e(y) \,dy}.
$$

The magnitude of the diffusion rate $\epsilon$ relative to the threshold $\epsilon^*(f)$ determines the sign of $u_\epsilon$:
\begin{itemize}
    \item Case $\epsilon > \epsilon^*(f)$: The local diffusion dominates, ensuring $u_\epsilon > 0$ strictly inside $\Omega$ and $\partial_\eta u_\epsilon < 0$ on $\partial\Omega$.
    \item Case $\epsilon = \epsilon^*(f)$: The solution remains nonnegative ($u_\epsilon \ge 0$ in $\Omega$), but strict positivity or the strict boundary inequality is lost. At least one of the following occurs: the solution vanishes at an interior point $x_0 \in \Omega$, or its normal derivative vanishes at a boundary point $x_1 \in \partial\Omega$.
    \item Case $\epsilon < \epsilon^*(f)$: The diffusion is insufficient to avoid the non-local effects. Consequently, the solution $u_\epsilon$ necessarily changes sign within $\Omega$.
\end{itemize}

From an ecological perspective, $u_\epsilon(x)$ can be interpreted as the steady-state density of a population subject to a spatial dispersal rate $\epsilon$ and global resource competition modeled by the integral term. The threshold $\epsilon^*(f)$ represents the minimum mobility compatible with a nonnegative state. If $\epsilon<\epsilon^*(f)$, the mathematical solution changes sign and therefore no longer represents a physically admissible population density.

\subsection{General nonlocal elliptic equation}
{\color{black}In this subsection, we study the non-local Dirichlet problem
\begin{equation}\label{general_a}
\left\{
\begin{array}{lll}
-\nabla\cdot(c \nabla u) + a\varphi \displaystyle\int_\Omega \psi(y) u(y) \,dy = f & \text{in}& \Omega,\\[1em]
u = 0 & \text{on}& \partial\Omega,
\end{array}
\right.
\end{equation}
where $f, \psi, \varphi \in L^\infty(\Omega)$, with $\varphi, \psi \ge 0$ a.e. in $\Omega$ and $\varphi, \psi \not\equiv 0$. The matrix $c \in L^\infty(\Omega; \mathbb{R}^{N \times N})$ is assumed to be uniformly elliptic, that is, there exists a constant $\gamma > 0$ such that
$$
c(x)\xi \cdot \xi \ge \gamma |\xi|^2 \quad \text{for a.e. } x \in \Omega \text{ and all } \xi \in \mathbb{R}^N.
$$
We define the sets
$$
\mathcal{B}_\varphi=\{f \in L^\infty(\Omega) : f \geq 0 \}
$$
and 
$$\mathcal{A}_\varphi = \mathcal{B}_\varphi \setminus \{r\varphi : r \ge 0\},$$
and we assume that $f\in \mathcal{A}_\varphi$.}

To analyze this non-local problem, we introduce $u_f$ and $e$ as the unique solutions to the corresponding local problems:
\begin{equation*}
\left\{
\begin{array}{lll}
-\nabla\cdot(c \nabla u_f) = f & \text{in}& \Omega,\\[1em]
u_f = 0 & \text{on}& \partial\Omega,
\end{array}
\right.
\quad \text{and} \quad
\left\{
\begin{array}{lll}
-\nabla\cdot(c \nabla e) = \varphi & \text{in}& \Omega,\\[1em]
e = 0 & \text{on}& \partial\Omega.
\end{array}
\right.
\end{equation*} 

By the strong maximum principle, $e > 0$ in $\Omega$, and by Hopf's Lemma, $\partial_\eta e < 0$ on $\partial \Omega$. Consequently, $$\int_\Omega \psi(y) e(y) \,dy > 0,$$ which ensures that $$1 + a \int_\Omega \psi e \,dy \ge 1,\quad \forall a \ge 0.$$ This allows us to represent the unique solution $u_a$   to the non-local problem \eqref{general_a}  explicitly as:
$$
u_a(x) = u_f(x) - \lambda(a, f, \Omega) e(x),
$$
where the parameter $\lambda(a, f, \Omega)$ is defined by:
$$
\lambda(a, f, \Omega) = \frac{\displaystyle a \int_\Omega \psi(y) u_f(y) \,dy}{\displaystyle 1 + a \int_\Omega \psi(y) e(y) \,dy}.
$$

To understand when $u_a$ remains strictly positive inside the domain, we introduce two critical parameters:
$$
\lambda^*(f) := \inf_{x \in \Omega} \frac{u_f(x)}{e(x)}\quad \text{and} \quad \lambda^{\ast\ast}(f) := \inf_{\partial\Omega} \frac{\partial_{\eta} u_f}{\partial_\eta e}.
$$
The quotient $u_f/e$ extends continuously to $\overline\Omega$, with boundary values $\partial_\eta u_f/\partial_\eta e$. If $f\in\mathcal A_\varphi$, the strong maximum principle and the Hopf lemma give $\lambda^*(f)>0$. Moreover, $u_f-\lambda^*(f)e$ is nonnegative and not identically zero, because equality would imply $f=\lambda^*(f)\varphi$.

Because $\psi \geq 0$ and $u_f - \lambda^*(f) e \gneq 0$, the integral 
$$
\int_\Omega \psi(y) \big( u_f(y) - \lambda^*(f) e(y) \big) \,dy
$$ 
is non-negative.
This naturally defines a critical  value $a^*(f)$ for the non-local parameter:
$$
a^*(f) :=
\left\{
\begin{array}{lll}
    \noalign{\smallskip}\displaystyle +\infty, \quad &\text{if}\displaystyle\quad \int_\Omega \psi(y) \big( u_f(y) - \lambda^*(f) e(y) \big) \,dy=0,\\[1em]
    \noalign{\smallskip}\displaystyle\frac{\lambda^*(f)}{\displaystyle \int_\Omega \psi(y) \big( u_f(y) - \lambda^*(f) e(y) \big) \,dy},\quad &\displaystyle\text{if}\quad \int_\Omega \psi(y) \big( u_f(y) - \lambda^*(f) e(y) \big) \,dy>0.
\end{array}
\right.
$$

When $$\int_\Omega \psi(y) \big( u_f(y) - \lambda^*(f) e(y) \big) \,dy > 0,$$ the parameter $a^*(f)$ distinguishes the behavior of the solution $u_a$ into three distinct regimes according to the validity of the maximum principle:
\begin{itemize}
    \item Case $0 \le a < a^*(f)$: here, $\lambda(a, f, \Omega) < \lambda^*(f)$. We can rewrite the solution as $u_a(x) = \big(u_f(x) - \lambda^*(f) e(x)\big) + \big(\lambda^*(f) - \lambda(a, f, \Omega)\big) e(x)$. Both terms are nonnegative, and the second is strictly positive in $\Omega$, ensuring $u_a > 0$ in $\Omega$. On the boundary, the first term has nonpositive outward normal derivative, while the second has strictly negative outward normal derivative. Hence ${\partial_\eta u_a} < 0$ everywhere on $\partial\Omega$.

    \item Case $a = a^*(f)$: here, $\lambda(a^*(f), f, \Omega) = \lambda^*(f)$, meaning $u_{a^*}(x) = u_f(x) - \lambda^*(f) e(x) \ge 0$. The strict positivity of the solution or its normal derivative begins to fail, depending on where the infimum $\lambda^*(f)$ is attained in $\overline{\Omega}$. If attained at an interior point $x_0 \in \Omega$, then $u_{a^*}(x_0) = 0$. If attained on the boundary at $x_0 \in \partial\Omega$, we have $\lambda^*=\lambda^{**}$ and ${\partial_\eta u_{a^*}}(x_0) = 0$, violating the strict boundary inequality of the Hopf Lemma at that point.

      \item Case $a > a^*(f)$: here $\lambda(a, f, \Omega) > \lambda^*(f)$. Assuming $u_f / e \in C(\overline{\Omega})$, let $x_0 \in \overline{\Omega}$ be a point where $\lambda^*(f)$ is attained:
    \begin{itemize}
        \item If $x_0 \in \Omega$, then $u_a(x_0) = e(x_0)\big( \lambda^*(f) - \lambda(a, f, \Omega) \big) < 0$. By continuity, $u_a < 0$ in an open neighborhood of $x_0$ of positive measure.
        \item If $x_0 \in \partial\Omega$, then ${\partial_\eta u_a}(x_0) = {\partial_\eta e}(x_0)\big( \lambda^*(f) - \lambda(a, f, \Omega) \big) > 0$. Since $u_a(x_0) = 0$ and $u_a$ increases strictly toward zero along the outward normal direction, $u_a < 0$ in a neighborhood of $x_0$ for points $x\neq x_0$.
    \end{itemize}
    However, $u_a$ cannot be nonpositive everywhere in $\Omega$. If $u_a \le 0$ in $\Omega$, then $$\int_\Omega \psi(y) u_a(y) \,dy \le 0,$$ which implies $$-\nabla\cdot(c(x) \nabla u_a) = f - a\varphi \int_\Omega \psi u_a \,dy \ge f \ge 0.$$ By the maximum principle, $u_a \ge 0$ in $\Omega$, forcing $u_a \equiv 0$ and thus $f \equiv 0$, which contradicts $f \in \mathcal{A}_\varphi$. Consequently, the solution $u_a$ must change sign in $\Omega$.
\end{itemize}


\end{document}